\documentclass[12pt,reqno]{amsart}
\usepackage{fullpage,amsfonts,amsmath,amssymb,rotating}
\numberwithin{equation}{section}
\usepackage{amsfonts,amsmath,amssymb,fullpage,graphicx}
\usepackage{verbatim}

\def\frak{\mathfrak}
\theoremstyle{plain} \newtheorem{Thm}{Theorem}[section]
\theoremstyle{plain} \newtheorem{Cor}[Thm]{Corollary}
\theoremstyle{plain} \newtheorem{Prop}[Thm]{Proposition}
\theoremstyle{plain} \newtheorem{Lemma}[Thm]{Lemma}
\theoremstyle{definition} 
\theoremstyle{definition} \newtheorem{Rem}[Thm]{Remark}
\theoremstyle{definition} 
\theoremstyle{definition} \newtheorem{Ex}[Thm]{Example}

\newcommand{\thmlist}{
\renewcommand{\theenumi}{\alph{enumi}}
\renewcommand{\labelenumi}{(\theenumi)}}
\renewcommand{\Re}{\mathop{\rm{Re}}}
\renewcommand{\Im}{\mathop{\rm{Im}}}

\newcommand{\norm}[1]{\|\/#1\/\|}

\newcommand{\normm}[1]{\|\/#1\/\|_m}
\newcommand{\inner}[2]{\langle#1,#2\rangle}
\newcommand{\innerm}[2]{\langle#1,#2\rangle_m}

\newcommand{\C}{\ensuremath{\mathbb C}}

\newcommand{\D}{\ensuremath{\mathbb D}}

\newcommand{\R}{\ensuremath{\mathbb R}}
\newcommand{\Z}{\ensuremath{\mathbb Z}}
\newcommand{\N}{\ensuremath{\mathbb N}}

\renewcommand{\l}{\lambda}
\renewcommand{\a}{\alpha}
\renewcommand{\b}{\beta}

\newcommand{\fa}{\mathfrak{a}}

\newcommand{\frakacs}{\frak a_{\C}^*}

\newcommand{\wrhob}{\widetilde{\rho}_\b}

\newcommand{\polya}{{\rm S}(\frak a_\C)}

\newcommand{\hyper}[4]{\ensuremath{\sideset{_{_2}}{_{_1}}{\mathop{F}}
\left(#1,#2;#3;#4\right)}}
\newcommand{\Jpoly}[4]{\ensuremath{P}_{#1}^{(#2,#3)}(#4)}

\newcommand{\Hreg}{A_\C^{\rm reg}}
\newcommand{\constA}{\mathbf{A}}
\newcommand{\constP}{\mathbf{P}}

\newcommand{\wt}{\widetilde}

\renewcommand{\phi}{\varphi}

\begin{document}

\makeatletter
\title[Ramanujan's Master Theorem on root systems]{Ramanujan's Master theorem for the hypergeometric Fourier transform on root systems}
\author{G. \'Olafsson}
\address{Department of Mathematics, Louisiana State University, Baton Rouge }
\email{olafsson@math.lsu.edu}
\author{A. Pasquale}
\address{Laboratoire de Math\'ematiques et Applications de Metz (UMR CNRS 7122),
Universit\'e de Lorraine, F-57045 Metz, France.}
\email{angela.pasquale@univ-lorraine.fr}
\thanks{The research of G. {\'O}lafsson was supported by NSF grant DMS-1101337.
A. Pasquale would like to thank the Louisiana State University, Baton Rouge, for
hospitality and financial support. She also gratefully acknowledges 
financial support from Tufts University and travel support from the Commission de Colloques et Congr\`es Internationaux (CCCI)}

\date{}

\begin{abstract}
Ramanujan's Master theorem states that, under suitable conditions, the Mellin transform of an alternating power 
series provides an interpolation formula for the coefficients of this series. Ramanujan applied this theorem 
to compute several definite integrals and power series, which explains why it is referred to as the ``Master 
Theorem''. In this paper we prove an analogue of Ramanujan's Master theorem for the hypergeometric Fourier
transform on root systems. This theorem generalizes to arbitrary positive multiplicity functions the results
previously proven by the same authors for the spherical Fourier transform on semisimple Riemannian symmetric spaces.
\end{abstract}

\maketitle

\section*{Introduction}
Ramanujan's First Quaterly Report \cite[p. 297]{Berndt} contains the following formal identity, nowadays known as 
Ramanujan's Master theorem: if a function $f(x)$ can be expanded
around $x=0$ in a power series of the form
$$f(x)=\sum_{k=0}^{\infty} (-1)^k a(k) x^k$$
then
\begin{equation}
\label{eq:Ramanujan}
\int_0^{+\infty} x^{-\l-1} f(x)\; dx= -\frac{\pi}{\sin(\pi \l)} \, a(\l)\,.
\end{equation}
By replacing $a(\l)$ with $A(\l)=a(\l)\Gamma(\l+1)$, one obtains  an equivalent version of (\ref{eq:Ramanujan}) as follows: 
 if a function $f(x)$ can be expanded
around $x=0$ in a power series of the form
$$f(x)=\sum_{k=0}^{\infty} (-1)^k \frac{A(k)}{k!} x^k$$
then
\begin{equation}
\label{eq:Ramanujan-gamma}
\int_0^{+\infty} x^{-\l-1} f(x)\; dx= \Gamma(-\l) \, A(\l)\,.
\end{equation}
Ramanujan's presented this formula as ``an instrument by which at least some of the definite integrals whose values are at present not known can be evaluated" (see \cite[p. 297]{Berndt}). As reported by Berndt \cite[p. 299]{Berndt}, ``Ramanujan was evidently quite fond of this clever, original technique, and he employed it in many contexts".  Most of the examples given  by Ramanujan by applying his Master Theorem turn out to be correct. But formulas (\ref{eq:Ramanujan}) and (\ref{eq:Ramanujan-gamma}) cannot hold without additional assumptions, as one can easily see from the example $a(\l)=\sin(\pi \l)$. The first rigorous reformulation of Ramanujan's Master theorem was given by Hardy 
in his book on Ramanujan's work \cite{Hardy}.  Using the residue theorem, Hardy proved that (\ref{eq:Ramanujan}) holds  for a natural class of functions $a$ and a natural set of parameters $\l$. 

Let $\constA$, $\constP$, $\delta$ be real constants so that $\constA <\pi$ and $0<\delta\leq 1$. Let
$\mathcal H(\delta)=\{\l \in \C: \Re \l> -\delta\}$. The Hardy class $\mathcal H(\constA,\constP,\delta)$ 
consists of all functions $a:\mathcal H(\delta) \to \C$ that are holomorphic on $\mathcal H(\delta)$
and satisfy the growth condition
$$|a(\l)| \leq C e^{-\constP(\Re \l)+\constA|\Im \l|}$$
for all $\l \in \mathcal H(\delta)$. 
Hardy's version of Ramanujan's Master theorem is the following, see \cite[p. 189]{Hardy}.

\begin{Thm}[Ramanujan's Master Theorem]
\label{thm:Ramanujan}
Suppose $a \in \mathcal H(\constA,\constP,\delta)$. Then:
\begin{enumerate}
\thmlist
\item The power series
\begin{equation} \label{eq:fpowersum}
f(x)=\sum_{k=0}^{\infty} (-1)^k a(k) x^k
\end{equation}
converges for $0<x<e^{\constP}$ and defines a real analytic function on this domain.
\item Let $0<\sigma<\delta$. For $0<x<e^{\constP}$ we have
\begin{equation}\label{eq:extensionPhi}
f(x)=\frac{1}{2\pi i} \, \int_{-\sigma-i\infty}^{-\sigma+i\infty} \frac{-\pi}{\sin(\pi \l)} \, a(\l) x^{\l} \; d\l\,.
\end{equation}
The integral on the right hand side of (\ref{eq:extensionPhi}) converges uniformly on compact subsets of $]0,+\infty[$ and is
independent of the choice of $\sigma$.
\item Formula (\ref{eq:Ramanujan}) holds for the extension of $f$ to $]0,+\infty[$ and for all $\l \in \C$ with $0<\Re \l<\delta$.
\end{enumerate}
\end{Thm}

The last part of Theorem \ref{thm:Ramanujan} is obtained from its second part by applying Mellin's inversion.  

Ramanujan's Master theorem has been extended to Riemannian symmetric spaces in duality by several authors. The rank-one semisimple case has been considered by Bertram in \cite{Bertram}. The starting point of Bertram's extension is the following group theoretic interpretation of (\ref{eq:Ramanujan}). The functions $x^\l$ ($\l \in \C$) are the spherical functions on $X_G=\R^+$ and the $x^k$ ($k \in \Z$) are the spherical functions on the torus $X_U=U(1)$. Both $X_G$ and $X_U$ can be realized as real forms of their complexification $X_\C=\C^*$. Hence  (\ref{eq:Ramanujan}) gives a relation between the compact and noncompact spherical trasforms of the restrictions to $X_U$ and $X_G$ of a ``good" function defined on $X_\C$. Bertram's version of Ramanujan's Master theorem was obtained by replacing the duality between $U(1)$ and $\R^+$ inside $\C^*$ with the duality between symmetric spaces of the compact type $X_U=U/K$ and of noncompact type $X_G=G/K$ inside their complexification $X_\C=G_\C/K_\C$. Following the same point of view, the authors of the present paper proved in \cite{OP-Ramanujan-JFA} an analogue of Ramanujan's Master theorem for (reductive) Riemannian symmetric spaces of arbitrary rank.
Some special classes of semisimple or reductive symmetric space situations have also been considered by 
Bertram \cite{BertramCR}, Ding, Gross and Richard \cite{DGR} and Ding \cite{Ding}. 

In this paper we consider Ramanujan's Master theorem on root systems. It is a generalization of our 
result concerning semisimple Riemannian symmetric spaces. 
The harmonic analysis associated with root systems is developed on a complex torus 
$A_\C=TA$ built up from a triple $(\mathfrak a,\Sigma,m)$ consisting of a finite dimensional 
Euclidean real vector space $\mathfrak a$, a root system $\Sigma$ in the dual $\mathfrak a^*$ of 
$\mathfrak a$, and a positive multiplicity function $m$ on $\Sigma$; 
see Section \ref{subsection:root-systems} for the precise definitions. 
The pair $(T,A)$ inside $A_\C$ plays the role of the pair $(X_U,X_G)$ of Riemannian symmetric spaces in 
duality inside their common complexication $X_\C$. The compact and noncompact spherical Fourier trasforms 
on $X_U$ and $X_G$ are respectively replaced by the Jacobi and hypergeometric Fourier transforms on $T$ and 
$A$.  Our theorem deals with (normalized) alternating series of Jacobi polynomials on $T$ of the form 
\begin{equation}
f(t)=\sum_{\mu \in P^+} (-1)^{|\mu|} a(\mu+\rho) \wt P_\mu(t)\,.
\end{equation}
Here the $\wt  P_\mu$ are suitable normalizations of the Heckman-Opdam's Jacobi polynomials (see
Section \ref{section:JacobiSeries}). They are parametrized by the set $P^+$ of 
positive restricted weights. Moreover, as in Ramanujan's Master Theorem \ref{thm:Ramanujan}, the 
coefficients $a(\mu+\rho)$ are obtained from a 
holomorphic function $a$ belonging to a certain Hardy class $\mathcal H(\constA, \constP, \delta)$
associated with $(\mathfrak a, \Sigma, m)$ and depending on three real parameters $\constA$, $\constP$ and $\delta$. The function $f$ is proved to define a $W$-invariant holomorphic function on a neighborhood of $T$ in $A_\C$. It then extends holomorphically to a neighborhood of $A$ in $A_\C$
by means of the inverse hypergeometric Fourier transform 
\begin{equation}\label{eq:vague-Rama-HO-first}
f(a)=\int_{\sigma+i\mathfrak a^*} a(\l)b(\l) F_\l(a) \, \frac{d\l}{c(\l)c(-\l)}\,,
\end{equation} 
where $F_\l(a)$ denotes the hypergeometric function of spectral parameter $\l$.
The analogue of Ramanujan's formula (\ref{eq:Ramanujan}) is then obtained from (\ref{eq:vague-Rama-HO-first}) using an inversion theorem for the hypergeometric Fourier transform. The function 
$b(\l)$ appearing on the right-hand side of (\ref{eq:vague-Rama-HO-first}) is a normalizing factor
and is used to produce singularities at the right points.
It plays the role of the function $\frac{i}{2} \frac{1}{\sin(\pi x)}$ occurring in the classical formula by Ramanujan. We refer the reader to Theorem \ref{thm:RamanujanHO} for the complete 
statement of Ramanujan's Master theorem for the hypergeometric Fourier transform and for the unexplained notation. 
 
Many tools needed to prove Ramanujan's Master theorem for the hypergeometric transform on root systems are 
known. For instance, the $L^p$-harmonic analysis of the hypergeometric Fourier transform on $A$ is known for 
$p=2$ from the fundamental work of Opdam \cite{OpdamActa}, and the case of $1\leq p<2$ has been recently 
developed  in \cite{NPP}.  Some results related to the geometric case of symmetric spaces need nevertheless 
to be extended.  Among them, the study of the holomorphic extension of Jacobi series inside $A_\C$ of 
(multivariate) Jacobi  series. This is a generalization of some results of Lassalle on holomorphic extensions of 
Laurent series on Riemannian symmetric spaces of the compact type; see \cite{Lassalle}.  
It is presented in Section \ref{section:Jacobi} below. 
Notice also that in \cite{OP-Ramanujan-JFA} several objects have been introduced or studied  using the 
classification of root multiplicities associated with semisimple Riemannian symmetric spaces. 
Their definition and analysis need to be generalized to the context of positive multiplicity functions. 
Among them we mention the function $d$ (which in the symmetric space case is the polynomial from Weyl's 
dimension formula), see Section \ref{section:d}, and the normalizing function $b$ occurring in Ramanujan's 
formula (\ref{eq:vague-Rama-HO-first}), which will studied in Section \ref{section:b}.
The proof of Theorem \ref{thm:RamanujanHO} is given in Section \ref{section:proof-Ramanujan-roots}. 
We will not work out all details, our aim being to prove what is necessary to apply the methods used in 
\cite{OP-Ramanujan-JFA}. The final section, Section  \ref{section:examples}, is devoted to some examples.

 \section{Notation and preliminaries} \label{section:notation}

\subsection{Root systems and their structure}
\label{subsection:root-systems}

In the following we shall work with triples $(\mathfrak a, \Sigma,m)$ where:
\begin{enumerate}
\renewcommand{\labelenumi}{{\tiny $\bullet$}}
\item $\mathfrak a$ is finite dimensional real Euclidean vector space, with inner product $\inner{\cdot}{\cdot}$,
\item $\Sigma$ is a (not necessarily reduced) root system in the dual $\mathfrak a^*$ of $\mathfrak a$, 
\item $m$ is a positive multiplicity function on $\Sigma$, that is a $W$-invariant function $m:\Sigma \to ]0,\infty[$.  
Here $W$ is the Weyl group of $\Sigma$.
\end{enumerate}
We shall write $m_\a:=m(\a)$ for $\a \in \Sigma$ and extend $m$ to $\mathfrak a^*$ by  setting $m_\a=0$ for $\a \notin \Sigma$. The \emph{rank} of the triple $(\mathfrak a,\Sigma,m)$ is the dimension of $\mathfrak a$, which we shall indicate by $l$. 
 We  remark  that  our  notation  is  based  on  the  theory  of symmetric  spaces  and differs  from  the  Heckman-Opdam's  notation  as  follows: the  root  system $R$  used  by Heckman  and  Opdam  is  related  to  our root  system  
$\Sigma$  by $R  =\{2\alpha: \a \in \Sigma\}$, and the  multiplicity  function $k$  in Heckman-Opdam's work is given by $k_{2\a} =  m_\a/2$. 
The triple $(\mathfrak a,\Sigma,m)$ is said to be \emph{geometric} if there exists a Riemannian symmetric space 
of the noncompact type with restricted root system $\Sigma$ such that $m_\a$ is the multiplicity of the 
root $\a$ for all $\a \in \Sigma$; see \cite{Araki} or \cite[Ch. X, Exercise F]{He1}.

Let $\Sigma^+$ be a choice of positive
roots in $\Sigma$.  We indicate by $\mathfrak a^+$ the open Weyl chamber in $\mathfrak a$ on which 
all elements of $\Sigma^+$ are strictly positive. We denote by $\Pi=\{\a_1, \dots, \a_l\}$ the set of 
simple roots associated with $\Sigma^+$. A root $\a \in \Sigma$ is said to be \emph{unmultipliable} if
$2\a \notin \Sigma$. We denote by $\Sigma_*$ the set of unmultipliable roots and by $\Sigma_*^+=
\Sigma_* \cap \Sigma^+$ the set of positive unmultipliable roots. 
We set 
\begin{equation}\label{eq:rho}
\rho=\frac{1}{2}\, \sum_{\alpha\in\Sigma^+} m_\alpha \alpha=
\frac{1}{2}\, \sum_{\b\in\Sigma_*^+} \big(\frac{m_{\b/2}}{2}+ m_\b\big)\b\,\,.
\end{equation}

For every non-zero $\l \in \mathfrak a^*$, let  $A_\l \in \mathfrak a$ be determined 
by $\l(H) = \inner{H}{A_\l}$  
for all  $H \in \mathfrak a$, and set $H_\l := 2 A_\l /\inner{A_\l}{A_\l}$.  
The complexification $\mathfrak a_\C := \mathfrak a \otimes_\R  \C$ 
of $\mathfrak a$ can be viewed as the Lie algebra of the complex torus $A_\C := \mathfrak a_\C /2\pi i\{H_\a : \a \in \Sigma\}$.  We write $\exp:\mathfrak a_\C \to A_\C$ for the exponential map, 
with multi-valued inverse $\log$. The real form $A := \exp \mathfrak a$ of $A_\C$ is an abelian subgroup of $A_\C$ with Lie algebra $\mathfrak a$ such that $\exp:\mathfrak a \to A$ is a diffeomorphism.  We set $A^+ := \exp \mathfrak a^+$. The polar decomposition of $A_\C$ is 
$A_\C=TA$, where $T=\exp(i\mathfrak a)$ is a compact torus with Lie algebra $i\mathfrak a$. 

We extend the inner product to $\mathfrak a^*$ by setting $\inner{\l}{\mu}=\inner{A_\l}{A_\mu}$.
Let $\frakacs$ be the space of all $\C$-linear functionals on $\mathfrak a$.
The $\C$-bilinear extension to $\frakacs$ and 
$\mathfrak a_\C$ of the inner products $\inner{\cdot}{\cdot}$ on
$\mathfrak a^*$ and $\frak a$ will also be denoted by $\inner{\cdot}{\cdot}$.
The action of $W$ extends to $\frak a$ by duality, 
to $\frakacs$ and $\mathfrak a_\C$ by $\C$-linearity, 
and to $A_\C$ and $A$ by 
the exponential map. Moreover, $W$ acts on functions $f$ on any of these spaces
by $(wf)(x):=f(w^{-1}x)$, $w \in W$. The set of $W$-invariant elements in a space of functions 
$\mathcal A$ will be indicated by $\mathcal A^W$. 

\subsection{The lattice of restricted weights}
For $\l \in \mathfrak a_\C^*$ and $\a \in \Sigma$ we shall employ the notation
\begin{equation}\label{eq:la} 
\l_\a=\frac{\inner{\l}{\a}}{\inner{\a}{\a}}\,.       
\end{equation}
We set $\mathfrak a^*_+=\{\l \in \mathfrak a_\C^*:\text{$\l_\a\geq 0$ for all $\a\in \Sigma^+$}\}$. 
The lattices $P$ and $P^+$ of restricted weights, respectively positive restricted weights, are
\begin{align}
\label{eq:P}
P=&\{\mu\in\mathfrak a^*:\text{$\mu_\a\in\Z$ for all $\a \in \Sigma^+$}\}\,,\\
\label{eq:Pplus}
P^+=&\{\mu\in\mathfrak a^*:\text{$\mu_\a\in\Z^+$ for all $\a \in \Sigma^+$}\}\,.
\end{align}
Here and in the following we shall employ the notation $\Z^+$ for the set on nonnegative integers $0,1,2,\dots$.
According to Helgason's Theorem (see \cite[Theorem 4.1, p. 535]{He2}), when $(\mathfrak a,\Sigma,m)$ 
is geometric, $P^+$ coincides with the set of restrictions to $\mathfrak a$ of the highest weights 
of the finite-dimensional irreducible $K$-spherical representations of $G$.

Let $\Pi=\{\a_1,\dots,\a_l\}$ be the basis of $\mathfrak a^*$ consisting of simple roots in $\Sigma^+$.
For $j=1,\dots,l$ set
\begin{equation}
\label{eq:betaj}
\beta_j=\begin{cases}
\a_j &\text{if $2\a_j\notin \Sigma$}\\
2\a_j &\text{if $2\a_j\in \Sigma$}
\end{cases}\,.
\end{equation}
Then $\{\beta_1,\dots,\beta_l\}$ is a basis of $\mathfrak a^*$ consisting of simple roots in $\Sigma_*^+$. The \emph{fundamental restricted weights} $\omega_1,\dots,\omega_l$ are defined by the conditions
\begin{equation}
\label{eq:omegaj}
(\omega_j)_{\beta_k}=\dfrac{\inner{\omega_j}{\beta_k}}{\inner{\b_k}{\beta_k}}=\delta_{jk}\,.
\end{equation}
Then $\Pi_*=\{\omega_1,\dots,\omega_l\}$ is a basis of $\mathfrak a^*$.
For $\l\in\mathfrak a^*_\C$ we have
\begin{equation}\label{eq:lambdaSigmaast}
\l=\sum_{j=1}^l \l_j \omega_j \qquad\text{with}\qquad \l_j:=\l_{\b_j}=\frac{\inner{\l}{\b_j}}{\inner{\b_j}{\b_j}} \,.
\end{equation}
Set
\begin{equation} \label{eq:defrhoj}
\rho=\sum_{j=1}^l \rho_j \omega_j\,.
\end{equation}
Since $\b_j$ is a multiple of a simple root, we have
\begin{equation}
\rho_j=\frac{1}{2} \left( m_{\b_j}+ \frac{m_{\b_j/2}}{2}\right)\,.
\end{equation}
For an arbitrary $\b \in \Sigma_*^+$, we set
\begin{equation}
\label{eq:wrhob}
\wrhob=\frac{1}{2} \left( m_{\b}+ \frac{m_{\b/2}}{2}\right)\,.
\end{equation}
Notice that $\wrhob=\rho_j=\rho_{\b_j}$ if $\b=\b_j$, but $\wrhob\neq \rho_\b$ in general.

 In the following, we denote by $\norm{\cdot}$ the $W$-invariant norm on $\mathfrak a$ which is induced by the inner product $\inner{\cdot}{\cdot}$. The same notation will also be employed for the corresponding norm on $\mathfrak a^*$. 
If $\l=\Re\l+i\Im\l\in \mathfrak a_\C^*$ with $\Re\l, \Im\l\in\mathfrak a^*$, then we set $\norm{\l}^2=\norm{\Re\l}^2+\norm{\Im\l}^2$.

We set 
\begin{equation}
\label{eq:Omega}
\Omega=\max_{j=1,\dots,l} \|\omega_j\|\,.
\end{equation}

The choice of $\Pi_*$ as a basis for $\mathfrak a^*$ (and $\mathfrak a_\C^*$) is related to the fact that 
$\mu \in P^+$ if and only if
$$\mu=\sum_{j=1}^l \mu_j \omega_j \qquad \text{with $\mu_j \in \Z^+$\,, $j=1,\dots,l$}\,.$$
We shall write for such an element:
\begin{equation}
|\mu|=\mu_1+\dots +\mu_l\,.
\end{equation}

If $\mu \in P$, then the exponential $e^\mu: A_\C \to \C$ defined by $e^{\mu}(h):= e^{\mu(\log h)}$ is
single valued. The $e^\mu$ with $\mu \in P$ are the algebraic characters of $A_\C$, and their 
$\C$-linear span coincides with the ring of regular functions $\C[A_\C]$ on the affine algebraic
variety $A_\C$. The set $A^{\rm reg}_\C := \{h \in  A_\C : \text{$e^{2\a(h)}= 1$
for all $\a \in \Sigma$}\}$ consists of the regular points of $A_\C$ for the action of $W$.
The algebra $\C[A^{\rm reg}_\C]$ of regular functions on $A^{\rm reg}_\C$
is the subalgebra of the quotient field of $\C[A_\C]$ generated by $\C[A_\C]$ and by
$1/(1-e^{-2\a})$ for $\a \in \Sigma^+$.

\subsection{$c$-functions}
\label{section:c}
For $\a \in \Sigma^+$ define
\begin{equation} \label{eq:calpha}
c_\alpha(\l)=
\frac{\Gamma\big(\l_\a+\frac{m_{\a/2}}{4}\big)}{\Gamma\big(\l_\a+\frac{m_{\a/2}}{4}+\frac{m_{\a}}{2}\big)}\,.
\end{equation}
Set $\Gamma^*(x)=\Gamma(1-x)^{-1}$ and let $c^*_\a$ be given by the same formula as $c_\a$ but 
with $\Gamma^*$ in place of $\Gamma$, i.e. 
\begin{equation} \label{eq:cstaralpha}
c^*_\alpha(\l)=
\frac{\Gamma\big(1-\l_\a-\frac{m_{\a/2}}{4}-\frac{m_{\a}}{2}\big)}{\Gamma \big(1-\l_\a-\frac{m_{\a/2}}{4}\big)}\,.
\end{equation}
Then, by the classical relation $\Gamma(z)\Gamma(1-z)=\frac{\pi}{\sin(\pi z)}$, we obtain
\begin{equation} \label{eq:calpha-star}
c^*_\alpha(\l)=S_\a(\l) c_\a(\l)
\end{equation}
where
\begin{equation} \label{eq:Salpha}
S_\a(\l)=\frac{\sin\big(\pi(\l_\a+\frac{m_{\a/2}}{4})\big)}
{\sin\big(\pi(\l_\a+\frac{m_{\a/2}}{4}+\frac{m_{\a}}{2})\big)}\,.
\end{equation}
Define (see e.g. \cite[(3.4.2) and (3.5.2)]{HS94})
\begin{align}
 \label{eq:wtc}
\wt c(\l)&=\prod_{\a\in\Sigma^+} c_\a(\l)\,,\\
\label{eq:wtcstar}
\wt c\,^*(\l)&=\prod_{\a\in\Sigma^+} c^*_\a(\l)\,,
\end{align}
and set
\begin{align}
\label{eq:c}
c(\l)&=c_{\rm HC} \; \wt c(\l)\,,\\
\label{eq:cstar}
c^*(\l)&=c_{\rm HC}^* \; \wt c\,^*(\l)\,,
\end{align}
where $c_{\rm HC}$ and $c_{\rm HC}^*$ are normalizing constants so that $c(\rho)=1$ and 
$c^*(-\rho)=1$. 
Recall the classical relation  $\sqrt{\pi}\;\Gamma(2z)=2^{2z-1} 
\;\Gamma(z) \Gamma\left(z + \frac{1}{2}\right)$ and that $\l_{\a/2}=2\l_\a$. Then we can write the functions $c$ and $c^*$ as
\begin{align} 
c(\l)&=c_{\rm HC} \, \prod_{\b \in \Sigma_*^+} \frac{\Gamma(2\l_\b)}{\Gamma\big(2\l_\b+\frac{m_{\b/2}}{4}\big)}\,
\frac{\Gamma\big(\l_\b+\frac{m_{\b/2}}{4}\big)}{\Gamma\big(\l_\b+\frac{m_{\b/2}}{4}+\frac{m_\b}{2}\big)}
\notag \\
\label{eq:c-Sigmastar}
&=c'_{\rm HC} \, \prod_{\b \in \Sigma_*^+} 
\frac{2^{-2\l_\b}\Gamma(2\l_\b)}{\Gamma\big(\l_\b+\frac{m_{\b/2}}{4}+\frac{1}{2}\big) \Gamma\big(\l_\b+\frac{m_{\b/2}}{4}+\frac{m_\b}{2}\big)}
\end{align}
and, similarly,
\begin{align}
c^*(\l)&=c^*_{\rm HC}  \, \prod_{\b \in \Sigma_*^+} 
\frac{\Gamma\big(1-2\l_\b-\frac{m_{\b/2}}{2}\big) \Gamma\big(1-\l_\b-\frac{m_{\b/2}}{4}-\frac{m_\b}{2}\big)}{\Gamma(1-2\l_\b)\Gamma\big(1-\l_\b-\frac{m_{\b/2}}{4} \big)}\\
\label{eq:cstar-Sigmastar}
&=(c^*_{\rm HC})' \, \prod_{\b \in \Sigma_*^+} 
\frac{\Gamma\big(-\l_\b-\frac{m_{\b/2}}{4}+\frac{1}{2}\big) \Gamma\big(1-\l_\b-\frac{m_{\b/2}}{4}-\frac{m_\b}{2}\big)}{2^{2\l_\b}\Gamma(1-2\l_\b)}\,.
\end{align}
In (\ref{eq:c-Sigmastar}) and (\ref{eq:cstar-Sigmastar}) we have
\begin{equation} 
c'_{\rm HC}=c_{\rm HC}  \prod_{\b \in \Sigma_*^+} \big( \pi^{1/2} 2^{1-m_{\b/2}/2}\big)
\qquad\text{and} \qquad
(c^*_{\rm HC})'=c^*_{\rm HC}  \prod_{\b \in \Sigma_*^+} \big( \pi^{-1/2} 2^{-m_{\b/2}/2}\big)\,.
\end{equation}

\subsection{Heckman-Opdam's hypergeometric functions}
Let $S(\mathfrak a_\C)$ denote the symmetric algebra over $\mathfrak a_\C$ considered as the space
of polynomial functions on $\mathfrak a^*_\C$. Every $p \in S(\mathfrak a_\C)$ defines a 
constant-coefficient differential operator $\partial(p)$ on $A_\C$ and on $\mathfrak a_\C$ such that 
$\partial(x)$ is the directional derivative in the direction of $x$ for all $x \in \mathfrak a$. 
The algebra of the differential operators $\partial(p)$ with $p \in S(\mathfrak a_\C)$ 
will also be indicated by $S(\mathfrak a_\C)$.
We denote by $\D(A^{\rm reg}_\C):=\C[A^{\rm reg}_\C] \otimes S(\mathfrak a_\C)$ 
the algebra of differential operators on $A_\C$ with coefficients in $\C[A^{\rm reg}_\C]$.
Let $L_A$ denote the Laplace operator on $A$ and set 
  \begin{equation}
    \label{eq:Laplm}
  L_m:=L_A+\sum_{\a \in \Sigma^+} m_\a \,\frac{1+e^{-2\a}}{1-e^{-2\a}} \;
     \partial(A_\a)\,.
  \end{equation}
Heckman and Opdam proved in \cite{HOpd1} that,  for the triple $(\frak a,\Sigma,m)$,  
the commutant $\D(\frak a,m,\Sigma):=\{Q \in \D(\Hreg)^W: L_mQ=Q_m\}$
of $L_m$ in $\D(\Hreg)^W$ 
is a commutative algebra which plays the role the commutative 
algebra of the radial components on $A^+$ of the 
invariant differential operators on a Riemannian symmetric space of  
noncompact type. For geometric triples $(\mathfrak a, \Sigma, m)$ the operator 
$L_m$ coincides with the radial component on $A^+$ of the Laplace operator on a Riemannian symmetric space of noncompact type.  
The algebra $\D(\frak a,m,\Sigma)$ can be constructed algebraically.
Indeed one has
\begin{equation*}
  \D(\frak a,m,\Sigma):=\{D_p: p \in \polya^W\},
\end{equation*}
where the differential operator $D_p$ can be explicitly given in terms of 
Cherednik operators (also called trigonometric Dunkl operators); see \cite{OpdamActa} or \cite{HeckBou}.

Let $\l\in\frak a^*_\C$ be fixed. The system of differential equations
\begin{equation}
  \label{eq:hypereq}
  D_p \varphi=p(\l)\varphi, \qquad p \in \polya^W,
\end{equation}
is called the \emph{hypergeometric system of differential equations with spectral parameter
$\l$} associated with the data $(\frak a,\Sigma,m)$.
The \emph{hypergeometric function of spectral
parameter $\l$} is the unique analytic $W$-invariant function
$F_\l(a)$ on $A$ which satisfies the system of
differential equations (\ref{eq:hypereq}) and which is normalized
by $F_\l(e)= 1$. Here $e=\exp 0$. In the geometric case, the function $F_\l$
agrees with Harish-Chandra's (elementary) spherical function of spectral
parameter $\l$. 

\begin{Ex}[The rank-one case] \label{ex:Frankone}
The rank-one case corresponds to triples $(\mathfrak a, \Sigma, m)$ in which
$\frak a$ is one dimensional. Then the set $\Sigma^+$ consists
at most of two elements: $\b$ and, possibly, $\b$/2. Fix $H_0 \in \mathfrak a$ such that $\b(H_0)=1$ and normalize the inner product $\inner{\cdot}{\cdot}$ on $\mathfrak a$ so that $\inner{H_0}{H_0}=1$. Then, in the notation of Section
\ref{subsection:root-systems}, $H_\b/2=A_\b=H_0$ and $\inner{\b}{\b}=\inner{A_\b}{A_\b}=1$. We identify $\mathfrak a$ and $\mathfrak a^* $ with $\R$ (and their complexifications $\mathfrak a_\C$ and $\mathfrak a_\C^*$ with $\C$)
by identifying $H_\b$ and $\b$ with $1$. Hence $\l=\l_\b \b \in \mathfrak a_\C^*$ is identified with $\l_\b \in \C$ and 
$H \in \mathfrak a_\C$ with $\b(H)/2 \in \C$. In the following we shall use the simplified notation $\l \in \C$ instead of 
$\l_\b \in \C$.
We also identify $A$ and $\mathfrak a\equiv \R$ by means of the 
exponential map. 
The Weyl chamber $\frak a^+$ coincides with the half-line $]0,+\infty[$.
The Weyl group $W$ reduces to $\{-1,1\}$ acting on $\R$ and $\C$ by multiplication.
The algebra $\D(\mathfrak a, \Sigma, m)$ is generated by $L_m$, and the hypergeometric differential
system with spectral parameter $\l \in \C$ is equivalent to the single Jacobi differential equation.
Heckman-Opdam's hypergeometric function $F_\l$ coincides with the Jacobi function of the first kind
\begin{equation*}
  F_\l(\exp H)=
\hyper{\tfrac{1}{2}\big(\tfrac{m_{\b/2}}{2}+m_\b\big)+\l}{\tfrac{1}{2}\big(\tfrac{m_{\b/2}}{2}+m_\b\big)-\l}{{\tfrac{m_{\b/2}+m_{\b}+1}{2}}}
{-\sinh^2 \tfrac{\b(H)}{2}}\,.
\end{equation*}
\end{Ex}

Schapira proved in \cite{Schapira} that $F_\l$ is real and strictly positive for $\l\in\frak a^*$. Moreover
\begin{equation} \label{eq:Schapira}
|F_\l|\leq F_{\Re\l}\,, \qquad \l \in \frak a_\C^*\,.
\end{equation}

For every $\l\in\mathfrak a_\C^*$ the hypergeometric function $F_\l$ extends holomorphically as a $W$-invariant function on the domain $\exp(2\Omega_\pi)$ in $A_\C$, where
\begin{equation}
\label{eq:Omegapi}
\Omega_\pi=\{H\in \mathfrak a_\C: \text{$|\a(\Im H)|<\pi/2$ for all $\a\in\Sigma$}\}\,.
\end{equation}
An elementary proof of this fact was given by J. Faraut at the conference ``Harmonic Analysis on Complex Homogeneous Domains and Lie Groups'', Rome, May 17–19, 2001. Faraut's argument has been reproduced in \cite[p. 26]{BOP}.

We shall need the following estimates of the holomorphically extended hypergeometric functions. 
Recall the constant $\Omega$ from (\ref{eq:Omega}).

\begin{Lemma}
\label{lemma:estFl}
There is a constant $C>0$ so that
\begin{equation}
\label{eq:OpdamEstimates}
|F_\l(\exp H)|\leq C  e^{-\min_{w \in W} \Im(w\l(H_2))+\max_{w\in W} \Re(w\l(H_1))}
\end{equation}
for all $\l \in \mathfrak a_\C^*$ and all
$H=H_1+iH_2 \in \overline{\Omega_\pi}$ with $H_1, H_2 \in \mathfrak a$.
In particular:
\begin{enumerate}
\thmlist
\item
For all $\l \in \mathfrak a^*_+ + i\mathfrak a^*$ and $H \in \mathfrak a$ we have
$$|F_\l(\exp H)|\leq C  e^{\Omega\|H\|(\sum_{j=1}^l \Re\l_j)}$$
\item
For all  $H\in \overline{\Omega_\pi}$ and $\l\in \mathfrak a^*$ we have
$$|F_\l(\exp H)|\leq C  e^{\|\Im H\|\|\Im \l\|}\,.$$
\end{enumerate}
\end{Lemma}
\begin{proof} (see \cite[Lemma 5.1]{OP-Ramanujan-JFA})
The estimates (\ref{eq:OpdamEstimates}) are due to Opdam; see \cite{OpdamActa}, Proposition 6.1(2) and Theorem 3.15. For (a), we can suppose by $W$-invariance that $H \in \overline{\mathfrak a^+}$.
In this case, for $\l \in \mathfrak a^*_+ + i\mathfrak a^*$, we have
\begin{equation*}
0\leq \Re\l(H)=\sum_{j=1}^l \Re\l_j \omega_j(H) \leq \Omega \|H\| \big(\sum_{j=1}^l \Re\l_j\big)\,.
\end{equation*}
Part (b) follows immediately from (\ref{eq:OpdamEstimates}).
\end{proof}

\subsection{The hypergeometric Fourier transform}
\label{subsection:hypergeomFourier}

Let $da$ denote a fixed normalization of the Haar measure on $A$. We associate with the triple $(\mathfrak a,\Sigma,m)$
the measure $d\mu(a)=\mu(a)\, da$ on $A$, where
\begin{equation}
 \label{eq:mu}
\mu(a)=\prod_{\a \in\Sigma^+} \big|e^{\a(H)}-e^{-\a(H)}\big|^{m_\a}\,, \qquad a=\exp(H)\,.
\end{equation}
Notice that when $(\fa, \Sigma, m)$ comes from a Riemannian symmetric space $G/K$,
then $d\mu$ is the component along $A$ of the Haar measure on $G$ with respect to the Cartan decomposition $G=KAK$.

The \emph{hypergeometric Fourier transform}
$\mathcal Ff=\wt f$ of a sufficiently regular $W$-invariant functions on $A$ is the $W$-invariant function 
on $i\mathfrak a^*$ defined by  
\begin{equation} \label{eq:ftilde}
(\mathcal F f)(\l)=\wt {f}(\l)=\int_A f(a)F_{-\l}(a)\; d\mu(a)\,.
\end{equation}

The Plancherel theorem states that the hypergeometric Fourier transform $\mathcal F$
is an isometry of $L^2(A,d\mu)^W$ onto $L^2(i\frak a^*, |W|^{-1}|c(\l)|^{-2}d\l)^W$.
Here $|W|$ denotes the order of the Weyl group $W$ and $c$ is the $c$-function (\ref{eq:c}).
Moreover, $\mathcal F$ has the following inversion formula, which holds for instance if 
$f \in L^p(A,d\mu)^W$, with $1 \leq p \leq 2$, and $\mathcal F f \in L^1(i\mathfrak a^*, |c(\l)|^{-2}d\l)^W$: 
for almost all $a \in A$ we have
\begin{equation}
\label{eq:inversionsphericalG}
f(a)=\frac{1}{|W|} \,\int_{i\mathfrak a^*} \widetilde{f}(\l) F_\l(a) \; \frac{d\l}{c(\l)c(-\l)}\,.
\end{equation}
See \cite[Theorem 5.4]{NPP}. 

Finally, we shall need the $W$-invariant $L^p$-Schwartz space isomorphism. 
For $1\leq p\leq 2$, the $L^p$-Schwartz space $\mathcal S^p(A)^W$ is the set of
all  $W$-invariant $C^\infty$ functions $f$ on $A$ such that for each $N\in\N_0$ and
$q\in \polya$,
\begin{equation}
\label{eq:Schwartzp}
\sup_{H\in \mathfrak a} (1+\norm{H})^N F_0(\exp H)^{-\frac 2p}|\partial(q)f(\exp H)|<\infty.
\end{equation}
Notice that $C_c^\infty(A)^W\subset \mathcal S^p(A)^W\subset L^p(A , d\mu)^W$.
Hence $\mathcal S^p(A)^W$ is dense in $L^p(A,d\mu)^W$. Moreover, $\mathcal S^p(A)^W$ is a
Frech\'et space with respect to the seminorms defined by the left-hand side of (\ref{eq:Schwartzp}).
Set $\epsilon_p=\frac 2p-1$. Let
$C(\epsilon_p\rho)$ be the convex hull in $\mathfrak a^*$ of the set $\{\epsilon_p
w\rho: w\in W\}$, and let $\fa_{\epsilon_p}^*=C(\epsilon_p\rho)+ i \fa^*$.
Notice that, for $p=1$, the set $\fa_{\epsilon_1}^*=C(\rho)+i\fa^*$ is precisely the
set of parameters $\l$ for which $F_\l$ is bounded; see \cite[Theorem 4.2]{NPP}.
Let $\mathcal S(\fa_{\epsilon_p}^\ast)^W$ be the set
of all $W$-invariant functions
$g:\fa_{\epsilon_p}^\ast\rightarrow \C$ which are
holomorphic in the interior of  $\fa_{\epsilon_p}^\ast$,
continuous on $\fa_{\epsilon_p}^\ast$ and satisfy  for
all $r\in \N_0$ and $s\in\polya$
\begin{equation}
 \label{eq:Schwartzpa}
\sup_{\lambda\in \fa_{\epsilon_p}^\ast}
(1+\norm{\lambda})^r\left|\partial(s)g(\lambda)\right|<\infty.
\end{equation}
 Then $\mathcal S(\fa_{\epsilon_p}^\ast)^W$ is a Fr\'{e}chet
algebra under pointwise multiplication and with the topology
induced by the seminorms defined by the left-hand side of (\ref{eq:Schwartzpa}).
Notice that when $p=2$, this space reduces to the usual 
space of Schwartz functions on $i\fa^\ast$.
The Schwartz space isomorphism theorem states that the hypergeometric Fourier 
transform is a topological isomorphism between $\mathcal S^p(A)^W$ and 
$\mathcal S(\fa_{\epsilon_p}^\ast)^W$. This
theorem was proved in \cite[Theorem 4.1]{Schapira} for
the case $p=2$ and in \cite[Theorem 5.6]{NPP} for the general case. 
(The proof is in fact an easy adaptation of Anker's method for the geometric case \cite{AnkerJFA}.)

\section{Jacobi polynomials and Jacobi series}
\label{section:Jacobi}

\subsection{Jacobi polynomials}
Recall that  $\C[A_\C]$ denotes the space of finite $\C$-linear
combinations of elements $e^{\mu}$ with $\mu \in P$, and let $\C[A_\C]^W$ be 
its subspace of $W$-invariant elements. 
Set 
$$\delta(m,t):=\prod_{\alpha \in \Sigma^+}
|e^{i\alpha(H)}-e^{-i\alpha(H)}|^{m_\alpha}$$ 
for $t=\exp(iH)$ with $H \in \mathfrak a$.
Define an inner product $\inner{\cdot}{\cdot}_m$ on $\C[A_\C]^W$ by
$$\inner{f}{g}_m:=\int_T f(t) \overline{g(t)} \; \delta(m,t) \,dt\,,$$
where $dt$ is the normalized Haar measure on $T$.

Recall from (\ref{eq:Pplus})  the lattice $P^+$ of positive restricted weights.  
The \emph{orbit sums}
\begin{equation}\label{eq:Mmu}
M_\mu:=\sum_{\nu \in W\mu} e^\nu
\end{equation}
form a basis of $\C[A_\C]^W$ as $\mu$ varies in $P^+$ because each
$W$-orbit in $P$ intersects $P^+$ in exactly one point. The
\emph{Jacobi polynomial} $P_\mu$ is the exponential polynomial
\begin{equation}\label{eq:Jacobipoly}
P_\mu:=\sum_{\mu\ge\nu \in P^+} c_{\mu\nu} M_\nu
\end{equation}
where the coefficients $c_{\mu\nu}$ are defined by the following
conditions, cf.\ \cite[\S 1.3]{HS94}, or \cite[\S 11]{MacdonaldSLC}:
\begin{enumerate}
\renewcommand{\theenumi}{\alph{enumi}}
\renewcommand{\labelenumi}{(\theenumi)}
\item
$c_{\mu\mu}=1$\,,
\item
$\inner{P_\mu}{M_\nu}_m=0$ for all $\nu \in P^+$ with
$\nu < \mu$.
\end{enumerate}
Condition (b) turns out to be equivalent to the fact that $P_\mu$ satisfies the 2nd order differential equation on $T$:
\begin{equation}
L_m^T P_\mu=-\inner{\mu+2\rho}{\mu} P_\mu
\end{equation}
where 
$$L_m^T=L_T+\sum_{\a\in\Sigma^+} m_\a \; \frac{1+e^{-2\a}}{1-e^{-2\a}}\;  \partial(iA_\a)\,,$$
and $L_T$ is the Euclidean Laplace-Beltrami operator on the compact torus $T$.
The fact that $P_\mu$ is an eigenfunction of $L_m^T$ yields recursion relations for the coefficients $c_{\mu\nu}$.
One can then deduce that $c_{\mu\nu}>0$ for all $\mu, \nu \in P^+$ with $\nu \leq \mu$. 
See \cite[pp. 34--35]{MacdonaldSLC}.

Observe that, by definition,  $P_\mu(t)$ extends holomorphically
to $A_\C$ as a function of $t$.
Moreover $\{P_\mu \mid \mu \in P^+\}$ is a basis for $\C[A_\C]^W$ which is orthogonal
with respect to the inner product
$\inner{\cdot}{\cdot}_m$ (cf.  \cite[Corollary 1.3.13]{HS94}).

The relation between Jacobi polynomials and Heckman-Opdam's hypergeometric functions 
is given by the following lemma.

\begin{Lemma}
Let $\mu \in P^+$. Then for all $a \in A$
\begin{equation} \label{eq:FandP}
F_{\mu+\rho}(a)=c(\mu+\rho) P_\mu(a)\,.
\end{equation}
Formula (\ref{eq:FandP}) provides a holomorphic extension of $F_{\mu+\rho}$ to $A_\C$.
Moreover, for all $\mu \in P^+$ and $w \in W$, we have
\begin{equation} \label{eq:symmP}
c(w(\mu+\rho)-\rho)P_{w(\mu+\rho)-\rho}=c(\mu)P_\mu\,.
\end{equation}
\end{Lemma}
\begin{proof}
Formula (\ref{eq:FandP}) is  \cite[(4.4.1)]{HS94}, and (\ref{eq:symmP}) is a consequence of (\ref{eq:FandP}) and the 
$W$-invariance of $F_\l$ in $\l \in \mathfrak a_\C^*$.
\end{proof}

\begin{Ex}[The rank-one case] \label{ex:Frankone-compact}
We keep the notation and identifications of Example \ref{ex:Frankone}. 
In particular, we have the identification of $P^+$ with $\Z^+$
so that $\mu=n\b \in P^+$ corresponds to $n\in\Z^+$. The polynomial 
\begin{equation}
F_{n\b+\rho}(\exp H)=
\hyper{n+m_\b+\tfrac{m_{\b/2}}{2}}{-n}{\tfrac{m_{\b/2}+m_\b+1}{2}}{\tfrac{1-\cos\b(H)}{2}}\,.
\end{equation}
is related to the classical \emph{Jacobi polynomial} $\Jpoly{n}{a}{b}{x}$ 
(see e.g. \cite[10.8(16)]{Er2}) by
$$
\binom{n+a}{n} F_{n\b+\rho}(\exp H)=\Jpoly{n}{a}{b}{\cos\beta(H)}
$$
where
$$a=(m_{\b/2}+m_\b-1)/2 \qquad\text{and} \qquad b=(m_\b-1)/2\,.$$
As a special instance, one finds the \emph{symmetric Jacobi polynomials} 
\begin{equation}
\label{eq:symmJacobi}
X_n^{(m-1)/2}(x)=\hyper{-n}{n+m}{\tfrac{m+1}{2}}{\tfrac{1-x}{2}}\,,
\end{equation}
which are constant multiples of the Jacobi polynomials $\Jpoly{n}{a}{a}{x}$ with $a=\frac{m-1}{2}$.
The polynomials  (\ref{eq:symmJacobi}) are hypergeometric functions corresponding 
 to a rank-one reduced root system $\Sigma$ with $m_\b=m$.
By selecting special values of $m$, one obtains specific classes of special polynomials, 
such as the \emph{Legendre polynomials}
$$P_n(x)=\Jpoly{n}{0}{0}{x}=X_n^{-1/2}(x)=\hyper{-n}{n+1}{1}{\tfrac{1-x}{2}}\,,$$
the \emph{Tchebichef polynomials of first kind}
$$T_n(x)=X_n^{-1/2}(x)=\hyper{-n}{n}{\tfrac{1}{2}}{\tfrac{1-x}{2}}\,,$$
and the \emph{Tchebichef polynomials of second kind}
$$U_n(x)=(n+1)X_n^{1/2}(x)=(n+1)\hyper{-n}{n+2}{\tfrac{3}{2}}{\tfrac{1-x}{2}}.$$
See \cite[formulas 10.10(3), 10.11(24), 10.11(25)]{Er2}
\end{Ex}  

\begin{Prop} \label{prop:estJacobi}
For $\mu \in P^+$,  the Jacobi polynomial $P_\mu(a)$ is real valued and positive for $a \in A$.
For all $h=ta$ with $t \in T$, $a \in A$, we have $|P_\mu(h)|\leq P_\mu(a)$.
Moreover, for all $a=\exp(H) \in \overline{A^+}$,
\begin{equation} \label{eq:est-P}
e^{\mu(H)} \leq P_\mu(a) \leq c(\mu+\rho)^{-1} e^{\mu(H)}\,.
\end{equation}
Similarly, the hypergeometric function $F_{\mu+\rho}(a)$ is real valued and positive for $a \in A$.
For all $h=ta$ with $t \in T$, $a \in A$, we have $|F_{\mu+\rho}(h)|\leq F_{\mu+\rho}(a)$.
Furthermore, for all $a=\exp(H) \in \overline{A^+}$,
\begin{equation} \label{eq:est-F-P}
 c(\mu+\rho) e^{\mu(H)} \leq F_{\mu+\rho}(a) \leq e^{\mu(H)}\,.
\end{equation}
\end{Prop}
\begin{proof}
Notice first that $c(\mu+\rho)>0$. Recall the defining formula (\ref{eq:Jacobipoly}) for $P_\mu$ and that $c_{\mu\nu}>0$ for all $\nu < \mu$ and $c_{\mu\mu}=1$.  By (\ref{eq:Mmu}), we have $M_\nu(e)=|W\nu|=|W|/|W_\nu|$, 
where $W_\nu=\{w \in W:w\nu=\nu\}$. Hence (\ref{eq:FandP}) and (\ref{eq:Jacobipoly}) yield
$$
c(\mu+\rho)  \sum_{\mu \geq \nu \in P^+} c_{\mu\nu} \tfrac{|W|}{|W_\nu|}=c(\mu+\rho) P_\mu(e)=F_{\mu+\rho}(e)=1\,.
$$
Moreover, $w\nu(H)\leq \nu(H) \leq \mu(H)$ for all $w \in W$, $\nu \in P^+$ with $\nu \leq \mu$ and 
$H \in \overline{\mathfrak a^+}$. So for $a=\exp(H) \in \overline{A^+}$ we have
$$0 \leq M_\nu(a) \leq \tfrac{|W|}{|W_\nu|} \, e^{\nu(H)} \leq \tfrac{|W|}{|W_\nu|} \,  e^{\mu(H)}\,.$$
Therefore 
$$
P_\mu(a) \leq \sum_{\mu \geq \nu \in P^+} c_{\mu\nu} \tfrac{|W|}{|W_\nu|} \,  e^{\mu(H)} = c(\mu+\rho) ^{-1}
e^{\mu(H)}\,.$$
Finally, 
$$
P_\mu(a)  = \sum_{\mu \geq \nu \in P^+} c_{\mu\nu} \sum_{\eta \in W\nu} e^{\eta(H)} \geq c_{\mu\mu} e^{\mu(H)}=e^{\mu(H)}\,.
$$
This proves the first part of the proposition. The results for $F_{\mu+\rho}$  are then an immediate consequence of  (\ref{eq:FandP}).
\end{proof}

\begin{Rem}
By using the non-symmetric hypergeometric functions, Schapira proved, more generally, that $F_\l(a)$ is real valued and positive for all $\l \in \mathfrak a^*$ and $a \in A$. See \cite[Lemma 3.1 and Corollary 3.1]{Schapira}. The estimate 
(\ref{eq:est-F-P}) is classical for the special case of spherical functions on Riemannian symmetric spaces of the compact type.
See e.g. \cite[Proposition IV.5.2]{Faraut}.
\end{Rem}

\subsection{The norm of the Jacobi polynomials}

Let $\wt c$ and $\wt c\,^*$ be as in Section \ref{section:c}. 
Set $\normm{f}=\innerm{f}{f}^{1/2}$ for the $L^2$-norm of
$f \in \C[A_\C]^W$. The $L^2$-norm of $P_\mu$ has been computed by Opdam. It is given by 
\begin{align} \label{eq:normPnu}
\normm{P_\mu}^2&=|W| \; \frac{\wt c\,^*(-\mu-\rho)}{\wt c(\mu+\rho)} \\
&=|W| \; \prod_{\a \in \Sigma^+} 
\frac{\Gamma\big(\mu_\a+\rho_\a+\frac{m_{\a/2}}{4}+\frac{m_{\a}}{2}\big)}{\Gamma\big(\mu_\a+\rho_\a+\frac{m_{\a/2}}{4}\big)}\,
\frac{\Gamma\big(1+\mu_\a+\rho_\a-\frac{m_{\a/2}}{4}-\frac{m_{\a}}{2}\big)}{\Gamma\big(1+\mu_\a+\rho_\a-\frac{m_{\a/2}}{4}\big)}\,.
\end{align}
See \cite{OpdamShift} or \cite[Theorem 3.5.5]{HS94}. 
Moreover, by \cite[(3.5.14)]{HS94}, 
\begin{equation} \label{eq:cardW}
|W|=\prod_{\a \in \Sigma^+} 
\frac{\big(\rho_\a+ \frac{m_{\a/2}}{4} +\frac{m_{\a}}{2}\big)}{\big(\rho_\a+ \frac{m_{\a/2}}{4}\big)}\,.
\end{equation}
Formula (\ref{eq:normPnu}) for $\mu=0$ yields
\begin{equation}\label{eq:Idelta}
I_\delta:=\int_T |\delta(m,t)| \, dt=\norm{P_0}^2_m
=|W|\, \frac{\wt c\,^*(-\rho)}{\wt c(\rho)}\,.
\end{equation}
It follows from  (\ref{eq:FandP}) and (\ref{eq:normPnu}) that the $L^2$-norm of $F_{\mu+\rho}$ is
\begin{align} 
\normm{F_{\mu+\rho}}^2&=|W| c_{\rm HC}^2 \; \wt c\,^*(-\mu-\rho) \wt c(\mu+\rho) \\
&=I_\delta \; \frac{ \wt c\,^*(-\mu-\rho) \wt c(\mu+\rho)}{\wt c\,^*(-\rho)\wt c^(\rho)}\\
&=I_\delta \; c^*(-\mu-\rho) c(\mu+\rho)\,.
\label{eq:norm-Fmu}
\end{align}

Let  $d(\mu)$ be the function on $P^+$ defined by means of Vretare's formula, see \cite[Theorem 9.10, p. 321]{He3}:
$$d(\mu)=\left.\frac{c(\l-\mu)c(-\l+\mu)}{c(\l)c(-\l)}\right|_{\l=\mu+\rho}=\frac{1}{ c(\mu+\rho)} 
\left.\frac{c(-\l+\mu)}{c(-\l)}\right|_{\l=\mu+\rho}\,,$$
where we have used that $c(\rho)=1$.
Notice that, by (\ref{eq:calpha-star}),
\begin{equation*}
\frac{c(-\l+\mu)}{c(-\l)}=\prod_{\a\in\Sigma^+} \frac{S_\a(-\l)}{S_\a(-\l+\mu)} \; \frac{c^*(-\l+\mu)}{c^*(-\l)}=
\frac{c^*(-\l+\mu)}{c^*(-\l)}
\end{equation*}
since $S_\a(-\l)=S_\a(-\l+\mu)$ for $\mu_\a \in \Z$.
Hence, since $c^*(-\rho)=1$, we have
\begin{equation} \label{eq:d-mu}
d(\mu)=\frac{1}{c(\mu+\rho) c^*(-\mu-\rho)}\,.
\end{equation}
Formula (\ref{eq:norm-Fmu}) becomes therefore
\begin{equation} \label{eq:normF-d}
\normm{F_{\mu+\rho}}^2=I_\delta \;  \frac{1}{d(\mu)}\,.
\end{equation}

\begin{Rem}
In the case of Riemannian symmetric spaces of the compact type $U/K$,  (\ref{eq:normF-d}) reduces to the 
classical formula (consequence of Schur's orthogonality relations) relating the $L^2$-norm of the spherical 
function $\psi_\mu=F_{\mu+\rho}$ on $U/K$ to the dimension $d(\mu)$ of the corresponding spherical 
representation, see for instance \cite[p. 146]{Faraut}. 
Indeed, let $du$ and $dt$ be respectively the invariant probability measures on $U$ and $T$.  
Then, for every $K$-biinvariant continuous function $f$ on $U$, we have
$\int_U f(u)\; du=C \int_T f(t) |\delta(m,t)| dt$ where $C=\big[\int_T  
|\delta(m,t)| dt\big]^{-1}=I_\delta^{-1}$. 
Hence  (\ref{eq:normF-d}) gives
$$\int_U |\psi_\mu(u)|^2 \, du= I_\delta^{-1} \int_T |\psi_\mu(t)|^2 |\delta(m,t)| dt=I_\delta^{-1} 
\normm{\psi_\mu}^2=d(\mu)^{-1}\,.$$
 \end{Rem}

\subsection{The Jacobi transform and Jacobi series}
\label{section:JacobiSeries}

The \emph{(normalized) Jacobi transform} of $f \in L^2(T)^W$ is the function 
$\widehat{f}:P^+ \to \C$ defined by
$$\widehat{f}(\mu):=I_\delta^{-1} \inner{f}{P_\mu}_m
=I_\delta^{-1} \int_T f(t)P_\mu(t^{-1})\, \delta(m,t)\; dt\,.$$
Here we have used the property that $\overline{P_\mu(t)}=P_\mu(t^{-1})$ for $t \in T$.
Observe that, by (\ref{eq:FandP}), 
\begin{equation}
c(\rho+\mu)\widehat{f}(\mu)=I_\delta^{-1}\innerm{f}{F_{\mu+\rho}}\,.
\end{equation}
The inversion formula is given by the \emph{Jacobi series}
\begin{align}
f&=I_\delta \sum_{\mu \in P^+} \widehat{f}(\mu) \dfrac{P_\mu}{\|P_\mu\|^2_m}\\
&=\sum_{\mu \in P^+} d(\mu) c(\mu+\rho) \widehat{f}(\mu) F_{\mu+\rho}   
 \label{eq:invJacobi}
\end{align}
with convergence in the sense of $L^2$. In (\ref{eq:invJacobi}) we have used (\ref{eq:FandP}) and (\ref{eq:normF-d}).

Set $B=\{H \in \mathfrak a: \norm{H}<1\}$. 
For $\varepsilon >0$ the $U$-invariant domain in $A_\C$ defined by $D_\varepsilon=T\exp(\varepsilon B)$ is a 
$W$-invariant neighborhood of $T$ in $A_\C$. In the following lemma we study the holomorphic extension of Jacobi series.
It generalizes a result proven by Lassalle in the geometric case; see \cite{Lassalle} or \cite[Proposition V.2.3]{Faraut}.

\begin{Lemma} \label{lemma:estimatesLassalle}
Let $G:P^+ \to \C$ and let $\varepsilon >0$. The Jacobi series 
\begin{equation} \label{eq:Jacobi-series-F}
\sum_{\mu \in P^+}  d(\mu) G(\mu) F_{\mu+\rho}(x)
\end{equation}
converges normally on compact subsets of $D_\varepsilon$ if and only if there is a constant $C>0$ and so that
\begin{equation}\label{eq:Gcond} 
|G(\mu)|\leq C e^{-\varepsilon \norm{\mu}}
\end{equation}
for all $\mu \in P^+$. In this case, its sum is a $W$-invariant holomorphic function on $D_\varepsilon$. 
\end{Lemma}
\begin{proof} 
Suppose that $G$ satisfies (\ref{eq:Gcond}). It is enough to prove the normal convergence of (\ref{eq:Jacobi-series-F}) on compact subsets of the form $\overline{D_r}=
T\exp(r\overline{B})$ where $\overline{B}=\{H \in \mathfrak a: \norm{H} \leq 1\}$ and $0<r<\varepsilon$.
Let $h=t\exp(rH)$ with $H \in \overline{B} \cap \overline{\mathfrak a^+}$. By Proposition \ref{prop:estJacobi}
we have 
$$|F_{\mu+\rho}(h)| \leq F_{\mu+\rho}(\exp(rH)) \leq e^{r\mu(H)} \leq e^{r\norm{\mu}}\,.$$
This estimate extends to $\overline{D_r}$ by $W$-invariance. 
Hence 
$$|d(\mu)G(\mu)F_{\mu+\rho}(h)|\leq d(\mu) e^{-(\varepsilon -r)\norm{\mu}} \,,$$
which implies the convergence of (\ref{eq:Jacobi-series-F}) since $d(\mu)$ has polynomial growth.

Conversely, suppose that the series (\ref{eq:Jacobi-series-F}) converges normally on the compact subsets of $D_\varepsilon$.
Let $0<\varepsilon'<\delta<\varepsilon$. By the normal convergence of  (\ref{eq:Jacobi-series-F}) on $D_\delta$, we have for all $H \in \exp(\delta B) \subset D_\delta$:
$$1 \geq |d(\mu) G(\mu)| F_{\mu+\rho}(\exp H)=\Big|\frac{G(\mu)}{c^*(-\mu-\rho)}\Big| 
P_{\mu}(\exp H).$$
Hence, by (\ref{eq:est-P}), for all $H \in \exp(\delta B)$,
$$|G(\mu)|\leq |c^*(-\mu-\rho)| e^{\mu(H)}\,.$$
Taking $H=-(\delta/\|\mu\|)A_\mu$, we obtain that 
$$|G(\mu)|\leq |c^*(-\mu-\rho)| e^{-\delta\|\mu\|} \leq C e^{-\varepsilon'\|\mu\|}$$
where $C$ is a constant independent of $\mu$ and $\varepsilon'$.
This implies (\ref{eq:Gcond}).
\end{proof}

\section{The function $d$}
\label{section:d}
We extend $d(\mu)$, $\mu \in P^+$, to a meromorphic function on $\mathfrak a_\C^*$  by means of 
(\ref{eq:d-mu}):
\begin{equation} \label{eq:d}
d(\l)=\frac{1}{c(\l+\rho) c^*(-\l-\rho)}
\end{equation}

When $(\mathfrak a,\Sigma,m)$ is geometric, $d$ coincides with the polynomial given by Weyl's 
dimension formula (written in terms of restricted roots). The polynomial nature of $d$ is precised by Lemma \ref{lemma:d} below.
In the following, we set 
\begin{equation} \label{eq:Pi}
\Pi(\l)=\prod_{\b\in \Sigma_*^+} \l_\b\,.
\end{equation} 
Recall also the notation $\wt\rho_\b =\tfrac{1}{2}\big(\tfrac{m_{\b/2}}{2}+m_\b\big)$ and the constant $L_\b$ from (\ref{eq:Lbeta})\,.

\begin{Lemma} \label{lemma:d}
\begin{enumerate}
\thmlist
\item We have 
\begin{align} \label{eq:d-rhoshift}
d(\l-\rho)&=\frac{1}{c(\l)c^*(-\l)} \notag\\
&=C_d \; \Pi(\l) \prod_{\b\in\Sigma_*^+}   
\frac{\Gamma\big( \l_\b +\frac{m_{\b/2}}{4}+\frac{1}{2}\big)
\Gamma( \l_\b +\wt \rho_\b)}
{\Gamma\big( \l_\b -\frac{m_{\b/2}}{4}+\frac{1}{2}\big)
\Gamma\big( \l_\b -\wt\rho_\b+1\big)}
\end{align}
where
\begin{equation} \label{eq:Cd}
C_d=\wt c(\rho) \wt c\,^*(-\rho)\, \prod_{\b \in \Sigma_*^+} 2^{m_{\b/2}}.
\end{equation}
\item If $(m_{\b/2})/2 \in \Z^+$ and $0\neq m_\b \in \Z^+$, then $d(\l)$ is a polynomial in 

$\l \in \mathfrak a_\C^*$. Explicitely, 
\begin{equation}
d(\l-\rho)=C_d 
\prod_{\b \in \Sigma_*^+} \Big(\l_\b 
\prod_{k=0}^{\frac{m_{\b/2}}{2}-1} \big[ \l_\b -\big( \tfrac{m_{\b/2}}{4}-\tfrac{1}{2} \big)+k\big]
\prod_{k=0}^{2 \wt \rho_\b-2} [ \l_\b -(\wt \rho_\b-1)+k]\Big).
\end{equation}
Here we adopt the convention that empty products are equal to $1$.
\item 
The function $d(\l)$ has polynomial growth. More precisely, set 
\begin{equation}\label{eq:M}
M=\sum_{\b\in \Sigma_*^+} (m_{\b/2}+m_\b)\,.
\end{equation}
Then, for every positive $\varepsilon>0$, there are constants $C_0$, $C_0'$ and $C_0''$ so that 
\begin{equation}\label{eq:est-d}
|d(\l-\rho)|\leq C_0 \prod_{\b \in \Sigma_*^+} (1+|\l_\b|)^{m_{\b/2}+m_\b} 
\leq C'_0 (1+\norm{\l})^M \leq C_0'' \prod_{j=1}^l (1+|\l_j|)^M
\end{equation}
for all $\l \in \mathfrak a_\C^*$ with $\max_{\b \in \Sigma_*^+}|\arg(\l_\b)| \leq \pi-\varepsilon$ 
and for all $\l \in \mathfrak a_\C^*$ with $\Re\l_\b \geq -L_\b+\varepsilon$ for all $\b \in 
\Sigma_*^+$. 
\end{enumerate}
\end{Lemma}
\begin{proof}
Part (a) is an immediate consequence of (\ref{eq:c-Sigmastar}) and (\ref{eq:cstar-Sigmastar}). 

Recall that, if $0\neq 2a \in \Z^+$, then
$$
\frac{\Gamma(\l_\b+a)}{\Gamma(\l_\b-a+1)}=\frac{\Gamma(\l_\b-a+1+(2a-1))}{\Gamma(\l_\b-a+1)}
=\prod_{k=0}^{2(a-1)} (\l_\b-a+1+k)\,.
$$
This proves (b) by taking 
$a=\tfrac{m_{\b/2}}{4}+\tfrac{1}{2}$ and $a=\wt\rho_\b$.

Finally, to prove (c), observe first that $|\l_\b|\leq \norm{\l}\norm{\b}^{-1}$ for all $\l\in\mathfrak a_\C^*$ and $\b \in \Sigma_*^+$. Hence, if $a>0$, then there is a constant $K\geq 0$ so 
that $(1+|\l_\a|)^a \leq K(1+\norm{\l})^a$ for all $\l \in \mathfrak a^*_\C$. Similarly, 
since $\l=\sum_ {j=1}^l \l_j \omega_j$, there is a constant $K_1>0$ so that 
$(1+\norm{\l}) \leq K_ 1 \prod_{j=1}^l (1+|\l_j|)$ for all $\l \in \mathfrak a^*_\C$.
Recall that
\begin{equation} \label{eq:asympt-ratio-gamma}
\lim_{|z|\to +\infty} \frac{\Gamma(z+a)}{\Gamma(z)} e^{-a\log z}=1 \qquad\text{for $|\arg z|\leq 
\pi -\varepsilon$\,,}
\end{equation}
where $a\in \C$ is fixed, $\log$ is the principal part of the logarithm, and $\varepsilon>0$; see e.g. 
\cite[Ch. IV, p. 151]{Titchmarsch}.
Let now $a>0$. Then for fixed constants $C_1>1$ and $\varepsilon>0$ there are constants $N_1>0$ and 
$C_2 \geq C_1$ (also depending on $a$) so that 
$$
\left| \frac{\Gamma(z+a)}{\Gamma(z-a+1)} \right|
\leq C_1 (1+|z-a+1|)^{2a-1} \leq  C_2 (1+|z|)^{2a-1}
$$
for all $z$ with $|z|\geq N_1$ and $|\arg z|\leq \pi-\varepsilon$.
The function  $ \frac{\Gamma(z+a)}{\Gamma(z-a+1)}$ is
holomorphic on $\Re z > -a$, so bounded on compact subsets of this domain. It is in particularly
bounded on $$\{z \in \C: \text{$|z|\leq N_1$ and $|\arg z| \leq \pi-\varepsilon$}\}$$ and (supposing without loss of generality that $\varepsilon<\pi/2$ so that the region is compact) on 
$$\{z \in \C: \text{$\Re z \geq -a +\varepsilon$ and $\pi \geq |\arg z| \geq \pi-\varepsilon$}\}\,.$$ It follows that there is a constant $C_3\geq C_2$ so that 
$$
\left| \frac{\Gamma(z+a)}{\Gamma(z-a+1)}\right| \leq  C_3 (1+|z|)^{2a-1}
$$
for all $z \in \C$ with $|\arg z|\leq \pi-\varepsilon$  and for all $z \in \C$ with $\Re z \geq -a+\varepsilon$.

Choose now $a=\tfrac{m_{\b/2}}{4}+\tfrac{1}{2}$. We get
$$
\left|\frac{\Gamma\big( \l_\b +\frac{m_{\b/2}}{4}+\frac{1}{2}\big)}
{\Gamma\big( \l_\b -\frac{m_{\b/2}}{4}+\frac{1}{2}\big)}\right| \leq  C'_3 (1+|\l_\b|)^{m_{\b/2}/2}
$$
for all $\l \in \mathfrak a_\C^*$ with $|\arg(\l_\b)|\leq \pi-\varepsilon$ and for all $\l \in \mathfrak a_\C^*$ 
with $\Re\l_\b \geq -\big( \tfrac{m_{\b/2}}{4}+\tfrac{1}{2}\big)+\varepsilon$. 
Choose then $a=\wt\rho_\b$. In this case, we obtain
$$
\left| \frac{\l_\b \Gamma( \l_\b +\wt \rho_\b)}{\Gamma( \l_\b-\wt\rho_\b+1)}\right|
 \leq  C_3'' (1+|\l_\b|)^{m_{\b/2}/2+m_\b-1} |\l_\b|
\leq  C_3'' (1+|\l_\b|)^{m_{\b/2}/2+m_\b}\,
 $$
for all $\l \in \mathfrak a_\C^*$ with $|\arg(\l_\b)|\leq \pi-\varepsilon$ and for all $\l \in \mathfrak a_\C^*$ 
with $\Re\l_\b \geq -\wt\rho_\b+\varepsilon$.
These estimates yield the claim for $d(\l-\rho)$.
\end{proof}

\begin{Rem}
By the infinitesimal classification of Riemannian symmetric spaces (see \cite{Araki} or \cite[Ch. X, Exercise F]{He1}), the condition 
in (b) is satified by all geometric triples $(\mathfrak a, \Sigma,m)$. This condition is also satisfied if $\Sigma$ is reduced
and $m_\b\in \Z$, for instance in the so-called even multiplicity case \cite{OPEven}. 
\end{Rem}

\section{The function $b$}
\label{section:b}
Let $b$ be the meromorphic function on $\mathfrak a_\C^*$ defined by the equality
\footnote{The constant $(i/2)^l$ of \cite[formulas (37) and (53)]{OP-Ramanujan-JFA} should be 
corrected as $2^{-l}$.}
\begin{equation}
\label{eq:b-first-formula}
\frac{b(\l)}{c(\l)c(-\l)}=2^{-l} d(\l-\rho)\; \prod_{j=1}^l \frac{1}{\sin\big(\pi(\l_j-\rho_j)\big)}
\end{equation}
Then 
\begin{equation}
\label{eq:b}
b(\l)=2^{-l}  \; \frac{c(-\l)}{c^*(-\l)}  \; \prod_{j=1}^l \frac{1}{\sin\big(\pi(\l_j-\rho_j)\big)}\,.
\end{equation}
Recall the functions $S_{\b/2}$ and $S_\b$ from (\ref{eq:Salpha}). We have for all $\b\in\Sigma_*^+$:
\begin{align*}
S_{\b/2}(\l)S_\b(\l)&=\frac{\sin(2\pi\l_\b)}{\sin\big(\pi(2\l_\b+\frac{m_{\b/2}}{2})\big)}
\frac{\sin\big(\pi\big(\l_\b+\frac{m_{\b/2}}{4}\big)\big)}{\sin(\pi(\l_\b+\wt\rho_\b))}\\
&=\frac{\sin(\pi\l_\b)\cos(\pi\l_\b)}
{\cos\big(\pi\big(\l_\b+\frac{m_{\b/2}}{4}\big)\big)\sin(\pi(\l_\b+\wt\rho_\b))} \,.
\end{align*}
It follows then from (\ref{eq:calpha-star}) that 
\begin{equation} \label{eq:bexplicit-one}
b(\l)=C_b \prod_{\b \in \Sigma_*^+} \frac{\cos\big(\pi\big(\l_\b-\frac{m_{\b/2}}{4}\big)\big)}
{\cos(\pi\l_\b)\sin(\pi\l_\b)} \prod_{\b \in \Sigma_*^+ \setminus \{\b_1,\dots,\b_l\}} \sin\big(\pi(\l_\b-\wt\rho_\b)\big).
\end{equation}
where 
\begin{equation}\label{eq:Cb}
C_b=2^{-l} \,\frac{c_{\rm HC}}{c^*_{\rm HC}}
= 2^{-l} \,\frac{\wt c\,^*(-\rho)}{\wt c(\rho)}\,.
\end{equation}
\begin{Rem}
\label{rem:b}
By classification (see e.g. \cite[Ch. X, Exercice F.4]{He1}), the parity of the geometric root multiplicities are distinguished in four different cases. They are reported in the following table together with 
the corresponding parity of $\wt \rho_\b$:

\smallskip

\begin{center}
\renewcommand{\arraystretch}{1.3}
\begin{tabular}{|c||c|c|c|}
\hline
     & $\frac{m_{\b/2}}{2}$ & $m_\b$ & $\wt\rho_\b=\frac{1}{2} \big(\frac{m_{\b/2}}{2}+m_\b\big)$ \\[.2em]
\hline
(a) & 0 & $\in 2\Z$ & $\in \Z$ \\[.1em]
\hline
(b) & 0 & $\in 2\Z+1$ & $\in \Z+\frac{1}{2}$ \\[.2em]
\hline
(c) & $\in 2\Z$ & $\in 2\Z$ + 1 & $\in \Z+\frac{1}{2}$ \\[.2em]
\hline
(d) & $\in 2\Z+1$ & $\in 2\Z+1$ & $\in \Z$ \\[.1em]
\hline
\end{tabular}
\end{center}
\medskip

\noindent The corresponding values of $b(\l)$, which have been determined in \cite[Remark 4.5]{OP-Ramanujan-JFA}, are
\begin{align}
\label{eq:bexplicit}
b(\l)=& C'_b \Big(\prod_{\stackrel{\b \in\Sigma_*^+\setminus\{\b_1,\dots,\b_l\}}{\text{cases (b) or (c)}}}   \cot(\pi\l_\b)\Big)
\Big( \prod_{\stackrel{\b \in\Sigma_*^+\setminus\{\b_1,\dots,\b_l\}}{\text{case (d)}}}  
\tan(\pi\l_\b) \Big)  \notag \\ 
&\times \Big( \prod_{\stackrel{j\in\{1,\dots,l\}}{\text{cases (a),(b) or (c)}}}   \frac{1}{\sin(\pi\l_j)}\Big)
\Big( \prod_{\stackrel{j\in\{1,\dots,l\}}{\text{case (d)}}} \!   \frac{1}{\cos(\pi\l_j)} \Big)
\end{align}
where $C'_b=\pm C_b$ and the sign depends on the parity of the multiplicities.
\end{Rem}

\begin{Lemma} \label{lemma:Pib-holo}
Set 
\begin{equation} \label{eq:TPi}
T_\Pi:=\{\l \in \mathfrak a_\C^*: \text{$|\Re\l_\b| <1/2$ for all $\b\in \Sigma_*^+$}\}\,.
\end{equation}
Let $\Pi(\l)$ be as in (\ref{eq:Pi}). Then $\Pi(\l)b(\l)$ is holomorphic on $T_\Pi$.
\end{Lemma}
\begin{proof}
This is immediate from (\ref{eq:bexplicit}).
\end{proof}

\begin{Rem}
According to the possible values of $m$, the function $\Pi(\l)b(\l)$ might be holomorphic on a larger tube domain. See \cite[Corollary 4.6]{OP-Ramanujan-JFA} for the geometric case situation.
\end{Rem}

\begin{Lemma}
\label{lemma:singularhyperplanes}
Let $C_d$ be the constant defined in (\ref{eq:Cd}). The function
\begin{align} 
\frac{b(\l)}{c(\l)c(-\l)}&=2^{-l} d(\l-\rho)\; \prod_{j=1}^l \frac{1}{\sin\big(\pi(\l_j-\rho_j)\big)} \notag \\
\label{eq:b-on-Plancherel}
&=C_d \; \Pi(\l) \left(\prod_{\b\in\Sigma_*^+} 
\frac{\Gamma\big( \l_\b +\frac{m_{\b/2}}{4}+\frac{1}{2}\big)
\Gamma( \l_\b +\wt\rho_\b)}
{\Gamma\big( \l_\b -\frac{m_{\b/2}}{4}+\frac{1}{2}\big)
\Gamma( \l_\b -\wt\rho_\b+1)} \right)
\left(\prod_{j=1}^l \frac{1}{\sin\big(\pi(\l_j-\rho_j)\big)}\right)
\end{align}
is meromorphic on $\mathfrak a_\C^*$. Its possible singularities are along the following hyperplanes:
\begin{align*}
&\mathcal H_{1,\b,k}:=\{\l \in \mathfrak a_\C^*: \l_\b=-\wt\rho_\b-k\} 
\quad\text{with}\quad \b \in \Sigma_*^+, k \in \Z^+\,,\\
&\mathcal H_{2,\b,k}:=\{\l \in \mathfrak a_\C^*:\l_\b=-\tfrac{m_{\b/2}}{4}-\tfrac{1}{2}-k\} \quad\text{with}\quad \b \in \Sigma_*^+, k \in \Z^+\,, \\
&\mathcal H_{j,k}:=\{\l \in \mathfrak a_\C^*: \l_j=\rho_j+k\}\quad\text{with}\quad j=1,\dots,l\,, k \in \Z^+ 
\quad\text{(simple poles)}.
\end{align*}
The occurrence and the order of the singularities along $\mathcal H_{1,\b,k}$ and $\mathcal H_{2,\b,k}$ depend on the root $\b$ (according to whether $\b=\b_j$ for some $j=1,\dots, l$ or not) and on the multiplicities of $\b/2$ and $\b$. 
The situation is summarized in Tables 1 and 2.
\end{Lemma}
\begin{proof}
The explicit formula for in Part (a) is an immediate consequence of (\ref{eq:d-rhoshift}). The hyperplanes $\mathcal H_{1,\b,k}$ and $\mathcal H_{2,\b,k}$ correspond respectively to the singularities of $\Gamma\big( \l_\b +\frac{m_{\b/2}}{4}+\frac{1}{2}\big)$ and $\Gamma( \l_\b +\wt\rho_\b)$, whereas $\mathcal H_{j,k}$ corresponds to those of $[\Gamma( \l_j -\rho_j+1) \sin(\pi(\l_j-\rho_j))]^{-1}$  (Recall that $\l_j=\l_{\b_j}$ and $\rho_j=\wt\rho_{\b_j}$). 

Some of the singularities along the hyperplanes $\mathcal H_{1,\b,k}$ and $\mathcal H_{2,\b,k}$ are cancelled by zeros coming from the gamma functions in the denominator of (\ref{eq:b-on-Plancherel}). 
More precisely, when $\b=\b_j$ (Table 1), the zeros of $\Gamma\big(\l_\b-\frac{m_{\b/2}}{4}+\frac{1}{2}\big)^{-1}$ remove the singularities along $\mathcal H_{1,\b,k}$ in case (1) and those along  $\mathcal H_{2,\b,k}$ in case (2).  
When $\b\neq \b_j$ (Table 2), the zeros of $\Gamma\big(\l_\b-\frac{m_{\b/2}}{4}+\frac{1}{2}\big)^{-1}$ also remove the singularities along $\mathcal H_{1,\b,k}$ in cases (1) and (2) and those along  $\mathcal H_{2,\b,k}$ in case (3). Moreover,
because of the absence of the $\sin$-functions at the denominator for $\b\neq \b_j$, additional cancellations occur in 
this case: the zeros of $\Gamma(\l_\b-\wt\rho_\b+1)^{-1}$ remove the singularities along $\mathcal H_{1,\b,k}$ in case (3)
and along  $\mathcal H_{2,\b,k}$ in cases (1) and (6).
\end{proof}

\begin{sidewaystable}[h]
\thispagestyle{empty}

\vskip 15truecm
\vskip 5truecm
\begin{center}

\renewcommand{\arraystretch}{1.3}
\begin{tabular}{|c||l||c|c||c|}
\hline
case & multiplicities & $\mathcal H_{1,\b,k}$ & $\mathcal H_{2,\b,k}$ & $l_\b$\\[.1em]
\hline
(1) & $\tfrac{m_{\b/2}}{2} \in \Z$ & no sing & simple poles & $-\wt\rho_\b$\\[.2em]
\hline
(2) & $\tfrac{m_{\b/2}}{2} \notin \Z\,, m_{\b/2}+m_\b \in 2\Z+1$ & simple poles & no sing &$-\big(\tfrac{m_{\b/2}}{4}+\tfrac{1}{2}\big)$\\[.2em] 
\hline
(3) & $\tfrac{m_{\b/2}}{2} \notin \Z\,, m_\b \in 2\Z+3$ & simple poles for $k=0,\dots, \tfrac{m_\b-3}{2}$ & double poles & $-\big(\tfrac{m_{\b/2}}{4}+\tfrac{1}{2}\big)$\\[.2em]
\hline
(4)  & $\tfrac{m_{\b/2}}{2} \notin \Z\,, m_\b=1$ & \multicolumn{2}{c|}{double poles ($\mathcal H_{1,\b,k}=\mathcal H_{2,\b,k}$)} &$-\big(\tfrac{m_{\b/2}}{4}+\tfrac{1}{2}\big)$ \\[.2em]
\hline
(5a)  &  $\tfrac{m_{\b/2}}{2} \notin \Z\,, m_{\b/2}+m_\b \notin 2\Z+1, m_\b\notin 2\Z+1, m_\b> 1$ & simple poles & simple poles &$-\big(\tfrac{m_{\b/2}}{4}+\tfrac{1}{2}\big)$ \\[.2em]
\hline
(5b)  &  $\tfrac{m_{\b/2}}{2} \notin \Z\,, m_{\b/2}+m_\b \notin 2\Z+1, m_\b\notin 2\Z+1, m_\b \leq 1$ & simple poles & simple poles& $-\wt\rho_\b$\\[.2em]
\hline
\end{tabular}
\medskip
\caption{Case $\b=\b_j$}

\bigskip

\renewcommand{\arraystretch}{1.3}
\begin{tabular}{|c||l||c|c||c|}
\hline
case & multiplicities & $\mathcal H_{1,\b,k}$ & $\mathcal H_{2,\b,k}$ & $l_\b$\\[.1em]
\hline
(1) & $\tfrac{m_{\b/2}}{2} \in \Z, m_\b \in \Z$ & no sing & no sing &$--$\\[.2em]
\hline
(2) & $\tfrac{m_{\b/2}}{2} \in \Z, m_\b \notin \Z$ & no sing & simple poles &$-\wt\rho_\b$ \\[.2em]
\hline
(3) & $\tfrac{m_{\b/2}}{2} \notin \Z\,, m_{\b/2}+m_\b \in 2\Z+1$ & no sing  & no sing &$--$ \\[.2em] 
\hline
(4) & $\tfrac{m_{\b/2}}{2} \notin \Z\,, m_\b \in 2\Z+3$ & simple poles for $k=0,\dots, \tfrac{m_\b-3}{2}$ & double poles  &$-\big(\tfrac{m_{\b/2}}{4}+\tfrac{1}{2}\big)$\\[.2em]
\hline
(5)  & $\tfrac{m_{\b/2}}{2} \notin \Z\,, m_\b=1$ & \multicolumn{2}{c|}{double poles ($\mathcal H_{1,\b,k}=\mathcal H_{2,\b,k}$)}  &$-\big(\tfrac{m_{\b/2}}{4}+\tfrac{1}{2}\big)$ \\[.2em]
\hline
(6)  &  $\tfrac{m_{\b/2}}{2} \notin \Z\,, m_{\b/2}+m_\b \notin 2\Z+1, m_\b\notin 2\Z+1, 2\wt\rho_\b\in \Z$ & simple poles & no sing &$-\big(\tfrac{m_{\b/2}}{4}+\tfrac{1}{2}\big)$\\[.2em]
\hline
(7a)  &  $\tfrac{m_{\b/2}}{2} \notin \Z\,, m_{\b/2}+m_\b \notin 2\Z+1, m_\b\notin 2\Z+1, 2\wt\rho_\b\notin \Z, m_\b> 1$ & simple poles & sinple poles &$-\big(\tfrac{m_{\b/2}}{4}+\tfrac{1}{2}\big)$\\[.2em]
\hline
(7b)  &  $\tfrac{m_{\b/2}}{2} \notin \Z\,, m_{\b/2}+m_\b \notin 2\Z+1, m_\b\notin 2\Z+1, 2\wt\rho_\b\notin \Z, m_\b \leq 1$ & simple poles & sinple poles  &$-\wt\rho_\b$\\[.2em]
\hline
\end{tabular}
\medskip

\caption{Case $\b\neq \b_j$}
\end{center}
\end{sidewaystable}

In the last column of Tables 1 and 2 we have reported the first negative value $l_\b$ such that $\l_\b=l_\b$ is a singular hyperplane of $\frac{b(\l)}{c(\l)c(-\l)}$. Notice that $l_\b$ is not $W$-invariant since $b$ is not $W$-invariant. On the elements of $\Sigma_*^+$ inside a single $W$-orbit, there are nevertheless only at most two values $l_1\leq l_2<0$ for $l_\b$, and $l_2=l_{\b_j}$ where 
$\b_j$ is an element of the basis $\Pi_*$ belonging to this $W$-orbit. Notice also
that the geometric case is contained in case (1), i.e. $\frac{m_{\b/2}}{2} \in \Z$,  of the two tables. 
To measure the largest $W$-invariant tube domain around $i\mathfrak a^*$ on which $\frac{b(\l)}{c(\l)c(-\l)}$ is 
holomorphic,  we introduce the following constants.

For $\b \in \Sigma_*^+$ define
\begin{equation} \label{eq:Lbeta}
L_\b:=-l_{\b_j} \qquad \text{if $\b_j \in W\b \cap \Pi_*$}
\end{equation}
Hence $L_\b \in \big\{\wt \rho_\b, \frac{m_{\b/2}}{4}+\frac{1}{2} \big\}$.

Then only the singular hyperplanes 
$\mathcal H_{j,k}$, with $j=1,\dots,l$ and $k \in \Z^+$, intersect the region
\begin{equation} \label{eq:LSigma}
L_\Sigma:=\{\l \in \mathfrak a_\C^*: \text{$\Re\l_\b >-L_\b$ for all $\b\in\Sigma_*^+$}\}
\supset \overline{\mathfrak a^*_+}+i\mathfrak a^*
\,.
\end{equation} 

\begin{Cor} \label{cor:b-Plancherel-tube-holo}
For all $w \in W$ the function $\frac{b(w\l)}{c(\l)c(-\l)}$ is holomorphic in the tube domain 
\begin{equation} \label{eq:TSigma}
T_\Sigma=\{\l \in \mathfrak a_\C^*: \text{$|\Re\l_\b| <L_\b$ for all $\b\in\Sigma_*^+$}\}
\end{equation}
\end{Cor}

\begin{Rem}
In the geometric case we have $L_\beta=\wt\rho_\b$ for all $\b \in \Sigma_*^+$. Observe that 
if $L_\beta=\wt\rho_\b$ for all $\b \in \Sigma_*^+$, then the base of the tube $T_\Sigma$ is $C(\rho)^0$, the interior of the convex hull of $\{w\rho:w \in W\}$ in $\mathfrak a^*$. 
In this case, $T_\Sigma=C(\rho)^0+i \mathfrak a^*$ is the interior of the domain in $\mathfrak a_\C^*$ in which all hypergeometric functions $F_\l$ are bounded, see \cite[Theorem 4.2]{NPP}.
\end{Rem}

\subsection{Tube domains in $\mathfrak a_\C^*$}

Recall the constants $L_\b$ introduced in (\ref{eq:Lbeta}).
Let $m_L$ be the positive multiplicity function on $\Sigma$ defined for $\b\in \Sigma^*$ by
\begin{align} \label{eq:mL}
(m_L)_{\b/2}&:=m_{\b/2}\\
 (m_L)_{\b}&:=\begin{cases} m_\b  &\text{if $L_\b=\wt\rho_\b$}\\
\tfrac{1}{2}  &\text{if $L_\b= \tfrac{m_{\b/2}}{4}+\tfrac{1}{2}$}\,.
\end{cases}
\end{align}
The corresponding $\rho$-function is
\begin{equation}\label{eq:rhoL}
\rho_L=\tfrac{1}{2} \sum_{\b \in \Sigma_*^+} \left[\frac{(m_L)_{\b/2}}{2}+(m_L)_\b\right] \b=
\tfrac{1}{2} \sum_{\b \in \Sigma_*^+} L_\b \b\,.
\end{equation}
Hence $L_\b=(\wt{\rho_L})_\b$ in the notation of (\ref{eq:wrhob}). 

For $\delta>0$,  we consider the following tube domains in
$\mathfrak a^*_\C$ around the imaginary axis:
\begin{align}
\label{eq:Tdelta}
T_\delta&=\{\l\in\mathfrak a^*_\C: \text{$|\Re\l_\b|<\delta L_\b$ for all $\b\in\Sigma_*^+$}\}\,,\\
T'_{\delta}&=\{\l \in \mathfrak a_\C^*: \text{$|\Re\l_j|<\delta L_{\b_j}$ for all $j=1,\dots,l$}\}\,, \\
\label{eq:Trhodue}
T''_{\delta}&=\{\l \in \mathfrak a_\C^*: \text{$\Re\l_j<\delta L_{\b_j}$ for all $j=1,\dots,l$}\}\,.
\end{align}
The following lemma is standard.
A proof can be found for instance in \cite[Lemma 1.2]{OP-Ramanujan-JFA}.

\begin{Lemma}
\label{lemma:TLbeta}
Let $w_0$ be the longest element of $W$. Then
\begin{equation}\label{eq:Trhow0}
T'_\delta=T''_\delta \cap w_0(T''_\delta)
\end{equation}
and
\begin{equation} \label{eq:Trhouno}
T_{\delta}=\bigcap_{w \in W} w(T'_\delta)=\bigcap_{w \in W} w(T''_\delta)\,.
\end{equation}
In particular, $T_\delta$ is the largest $W$-invariant tube domain contained in $T'_{\delta}$.
Moreover,
\begin{equation} \label{eq:TrhoCrho}
T_\delta={\rm C}(\delta \rho_L)^0+i\mathfrak a^*
\end{equation}
where ${\rm C}(\nu)$ is the the convex hull of the $W$-orbit $\{w\nu: w \in W\}$ of $\nu\in \mathfrak a^*$ and ${\rm C}(\nu)^0$ is its interior.
\end{Lemma}

\begin{Rem}
Notice that $T_1=T_\Sigma$ is the tube domain introduced in Corollary \ref{cor:b-Plancherel-tube-holo}.
\end{Rem}

\section{Statement of Ramanujan's Master theorem for root systems}
\label{section:RamanujanHO}
\noindent
Let $A_\C=AT$ be the complex torus associated with a triple $(\mathfrak a,\Sigma,m)$ as in Section
\ref{subsection:root-systems}, and let $L_\b$, with $\beta \in \Sigma_*^+$, be defined by
(\ref{eq:Lbeta}). 
Let $\constA$, $\constP$, $\delta$ be constants so that $\constA <\pi$, $\constP>0$ and
$0<\delta\leq 1$, and define \begin{equation}\label{eq:Hdelta}
\mathcal H(\delta)=\{\l \in \mathfrak a_\C^*: \text{$\Re \l_\b > -\delta\,\wt\rho_\b$ for all $\b\in\Sigma_*^+$}\}\,.
\end{equation}
The \emph{Hardy class} $\mathcal H(\constA,\constP,\delta)$ consists of the functions $a:\mathcal H(\delta) \to \C$ that are holomorphic on $\mathcal H(\delta)$
and so that
\begin{equation}
\label{eq:defHardyclass}
|a(\l)| \leq C \prod_{j=1}^l e^{-\constP(\Re \l_j)+\constA|\Im \l_j|}
\end{equation}
for some constant $C \geq 0$ and for all $\l \in \mathcal H(\delta)$.

\begin{Thm}[Ramanujan's Master Theorem for root systems]
\label{thm:RamanujanHO}
Let $d$ and $b$ be the meromorphic functions on $\mathfrak a_\C^*$ defined by (\ref {eq:d}) and
(\ref{eq:b}), respectively. 
Suppose $a \in \mathcal H(\constA,\constP,\delta)$.
\begin{enumerate}
\item
Let $\Omega$ be as in (\ref{eq:Omega}).
Then the alternating normalized Jacobi series
\begin{equation} \label{eq:fFourier}
f(t)=\sum_{\mu\in P^+} (-1)^{|\mu|} d(\mu) a(\mu+\rho) F_{\mu+\rho}(t)
\end{equation}
converges normally on  compact subsets of $D_{\constP/\Omega}=T \exp\big((\constP/\Omega)B\big)$
where $B=\{H \in \mathfrak a: \|H\|<1\}$ is the open unit ball in $\mathfrak a$. Its sum is a $W$-invariant holomorphic function on the neighborhood $D_{\constP/\Omega}$ of $T$ in $A_\C$.

\item Let $T_\delta$ be the tube domain in (\ref{eq:Tdelta}) and let
$\sigma\in T_\delta \cap \mathfrak a^*$. Then for $x=\exp H \in A$ with $\|H\|<\constP/\Omega$, we have
\begin{equation}\label{eq:extensionf-A}
f(x)=\frac{1}{|W|} \int_{\sigma+i\mathfrak a^*} \left( \sum_{w \in W} a(w\l)b(w\l)\right) F_\l(x) \; \frac{d\l}{c(\l)c(-\l)} \,.
\end{equation}
The integral on the right-hand side of (\ref{eq:extensionf-A}) is independent of the choice of $\sigma$. It converges uniformly on compact subsets of $A$ and extends to a holomorphic $W$-invariant function on a neighborhood of $A$ in $A_\C$.
\item The extension of $f$ to $A$
satisfies
$$\int_A |f(x)|^2 \, d\mu(x) =\frac{1}{|W|} \int_{i\mathfrak a^*} \Big| \sum_{w \in W} a(w\l)b(w\l)\Big|^2 \; \frac{d\l}{|c(\l)|^2}
\,.$$
Moreover,
\begin{equation} \label{eq:RamanujanHO}
\int_A f(x)F_{-\l}(x)\; d\mu(x)=  \sum_{w \in W} a(w\l)b(w\l)
\end{equation}
for all $\l \in T_\delta \cap T_\Pi$. More precisely, the integral on the left-hand side of (\ref{eq:RamanujanHO}) converges in $L^2$-sense and absolutely on $i\mathfrak a^*$.
It defines a $W$-invariant holomorphic function on a $W$-invariant tube domain around $i\mathfrak a^*$, and (\ref{eq:RamanujanHO}) extends as an identity between holomorphic functions on $T_\delta \cap T_\Pi$.
\end{enumerate}
\end{Thm}

\begin{Rem}
As in the classical case (\ref{eq:Ramanujan-gamma}) or in the case of semisimple Riemannian symmetric spaces
in \cite{OP-Ramanujan-JFA}, there is an 
equivalent formulation of Ramanujan's Master theorem for root systems using the gamma function.
This version, as well as some immediate consequences, can be easily obtained, as in \cite[Remark 2.6]{OP-Ramanujan-JFA}, from Theorem \ref{thm:RamanujanHO}.
\end{Rem}

\section{Proof of Theorem \ref{thm:RamanujanHO}}
\label{section:proof-Ramanujan-roots}

\noindent The proof of Ramanujan's Master Theorem \ref{thm:RamanujanHO} for root systems follows the same pattern used 
in the proof of the corresponding theorem for semisimple Riemannian symmetric spaces 
\cite[Theorem 2.1]{OP-Ramanujan-JFA}. The first statement is an application of Lemma \ref{lemma:estimatesLassalle}.
The crucial step for the remaining parts, and this is the content of this section, consists in extending 
the necessary estimates from \cite{OP-Ramanujan-JFA} to the setting of positive multiplicity functions.  

Observe first that there is a constant $K>0$ so that
\begin{equation}
\label{eq:est-sin}
\big|\sin\big(\pi(\l_j-\rho_j)\big)\big|^{-1} \leq K e^{-\pi|\Im \l_j|}
\end{equation}
for $|\Im \l_j|\geq 1$ or for $\Re\l_j=\rho_j+N+1/2$ with $N\in \Z^+$.

The second part of Theorem \ref{thm:RamanujanHO} is proven by computing the integral on the right-hand side of 
(\ref{eq:extensionf-A}) by separately applying the Residue Theorem to each veriable $\l_j$. The term of parameter
$\mu+\rho$ in the alternating normalized Jacobi series appears as the result of taking residues 
(for $j=1,2,\dots,l$) in $\l_j$ at the singularity $\rho_j+\mu_j$ corresponding to the singular hyperplane 
$\mathcal H_{j,\mu_j}$ from Lemma \ref{lemma:singularhyperplanes}. The residues are computed using 
(\ref{eq:b-first-formula}). The following lemma contains the estimates needed to apply the Residue Theorem. 
It is an easy generalization of \cite[Lemma 5.3]{OP-Ramanujan-JFA}, and its proof is omitted.

\begin{Lemma}
\label{lemma:estimates-residues}
\begin{enumerate}
\thmlist
\item
Let $N$ be a positive integer and let $M$ be as in (\ref{eq:M}).
Let $\l=\sum_{j=1}^l \l_j \omega_j \in\mathfrak a_\C^*$ with
$|\Im \l_j|\geq 1$ or $\Re\l_j=\rho_j+N+1/2$ or $\Re\l_j=0$ for all $j=1,\dots,l$.
Then there is a positive constant $C_1$, independent of $N$, so that
\begin{equation*}
\left|\frac{b(\l)}{c(\l)c(-\l)}\right| \leq C_1 \prod_{j=1}^l \left[(1+|\l_j|)^M e^{-\pi |\Im \l_j|}\right]\,.
\end{equation*}
\item
Set
\begin{multline}
\label{eq:B}
B=\Big\{\l=\sum_{j=1}^l \l_j \omega_j \in\mathfrak a_+^*+i\mathfrak a^*:
\text{$|\Im \l_j|\geq 1$ or}\\ \text{$\Re\l_j \in (\rho_j+\Z^+ +1/2) \cup \{0\}$ for all $j=1,\dots,l$}\Big\}.
\end{multline}
Let $a \in \mathcal H(\constA,\constP,\delta)$. Then there is a constant $C_2>0$ so that for all $\l\in B$ and $H \in \overline{\mathfrak a^+}$ we have
\begin{equation}
\label{eq:main-est-shift}
\left|\frac{a(\l)b(\l)}{c(\l)c(-\l)} F_\l(\exp H) \right|
\leq C_2 \prod_{j=1}^l \left[ (1+|\l_j|)^M e^{(\constA-\pi)|\Im \l_j| +(\|H\|\Omega-\constP) \Re\l_j} \right]\,.
\end{equation}
\end{enumerate}
\end{Lemma}

The estimates in the next lemma allow us to prove that the integral on the right-hand side of 
(\ref{eq:extensionf-A}) is independent of the choice of $\sigma \in T_\delta \cap \mathfrak a^*$.

\begin{Lemma}
\label{lemma:estimatesTdelta}
Let $0<\delta \leq 1$ and let $T_\delta$ be the tube domain from (\ref{eq:Tdelta}). Let $M$ be the
constant defined in (\ref{eq:M}).
\begin{enumerate}
\thmlist
\item
There is a constant $C_\delta >0$ so that
\begin{equation}
\label{eq:estimatebcTdelta}
\left|\frac{b(\l)}{c(\l)c(-\l)} \right| \leq C_\delta (1+\|\l\|)^M e^{-\pi\big(\sum_{j=1}^l
|\Im \l_j|\big)}
\end{equation}
for all $\l \in T_\delta$.
\item
Let $a \in \mathcal H(\constA,\constP,\delta)$. 
For every $R>0$ and every integer $N\geq 0$ there is a constant $C_{R,N,\delta} >0$ so that for all $\l\in T_\delta$ and $H \in \mathfrak a$ with $\|H\|< R$, we have
\begin{equation}
\label{eq:estabphilN}
\left|\frac{a(\l)b(\l)}{c(\l)c(-\l)}F_\l(\exp H) \right|
\leq
C_{R,N,\delta} (1+\|\l\|)^{-N}
\end{equation}
Consequently,
\begin{equation}
\label{eq:estWabphilN}
\left|\big(\sum_{w\in W} a(w\l)b(w\l)\big)\frac{F_\l(\exp H)}{c(\l)c(-\l)} \right|
\leq
C_{R,N,\delta} |W|(1+\|\l\|)^{-N}\,.
\end{equation}
\end{enumerate}
\end{Lemma}
\begin{proof}
By Corollary \ref{cor:b-Plancherel-tube-holo}, for all $w \in W$, the function $\frac{b(w\l)}{c(\l)c(-\l)}$
is holomorphic on the $W$-invariant tube domain $T_\delta$. Set 
$$p_j(\l)=\big(\l_j-(\rho_j-1)\big)\big(\l_j-(\rho_j-2)\big)\cdots \big(\l_j-(\rho_j-h_j)\big)$$ 
where $h_j \in \Z^+$ is chosen so that $2\rho_j-1 \leq h_j < 2\rho_j$. Then, by (\ref{eq:b-on-Plancherel}), 
\begin{multline}
\label{eq:bontube}
\frac{b(\l)}{c(\l)c(-\l)}=C_d \; \Pi(\l) \left(\prod_{\b\in\Sigma_*^+} 
\frac{\Gamma\big( \l_\b +\frac{m_{\b/2}}{4}+\frac{1}{2}\big)}
{\Gamma\big( \l_\b -\frac{m_{\b/2}}{4}+\frac{1}{2}\big)}\right) \times \\
 \left(\prod_{\b\in\Sigma_*^+\setminus\{\b_1,\dots,\b_l\}}  
\frac{\Gamma( \l_\b +\wt\rho_\b)}{\Gamma( \l_\b -\wt\rho_\b+1)} \right)
\left(\prod_{j=1}^l \frac{\Gamma(\l_j+\rho_j)}{\Gamma(\l_j -(\rho_j-h_j)+1)} 
\frac{p_j(\l)}{\sin\big(\pi(\l_j-\rho_j)\big)}\right)
\end{multline}
For fixed $\eta \in ]0,1[$, the function $\frac{z}{\sin(\pi z)}$
is bounded on $\{z \in \C: |\Im z|\leq 1, |\Re z|\leq \eta\}$. 
By (\ref{eq:est-sin}), we conclude that there is a constant 
$C'_\delta>0$ so that for any fixed $j=1,\dots, l$ and every $\l=\sum_h \l_h \omega_h$ with 
$|\Re\l_j|\leq \delta L_{\b_j}$ and arbitrary $\l_h \in \C$ with $h\neq j$, we have
$$\left| \frac{p_j(\l)}{\sin\big(\pi(\l_j-\rho_j)\big)}\right| \leq C'\delta (1+|\l_j|)^{\deg p_j} e^{-\pi|\Im \l_j|}\,.$$
Part (a) then follows from these estimates and (\ref{eq:bontube}).

To prove (b), observe first that $T_\delta$ is a $W$-invariant subset of $\mathcal H(\delta)$. 
Hence $$\big(\sum_{w\in W} a(w\l)b(w\l)\big)\frac{F_\l(\exp H)}{c(\l)c(-\l)}$$ is holomorphic on $T_\delta$
for every fixed $a\in\mathcal H(A,P,\delta)$.
Let $R>0$ be fixed. By (\ref{eq:Schapira}) there is a constant $C_{R,\delta}>0$ so that
\begin{equation}
\label{eq:estphilcomp}
|F_\l(\exp H)| \leq C_{R,\delta}
\end{equation}
for all $\l \in T_\delta$ and $H\in \mathfrak a$ with $\|H\| \leq R$.
(In fact, since $T_\delta \subset T_1$, one knows from \cite[Theorem 4.2]{NPP} that $C_{R,\delta}$ can 
be chosen to be equal to $1$.)
Together with Part (a), (\ref{eq:estphilcomp}) yields that there is a constant
$C'_{R,\delta}>0$ so that
\begin{equation*}
\left|\frac{a(\l)b(\l)}{c(\l)c(-\l)} F_\l(\exp H) \right|
\leq
C'_{R,\delta} (1+\|\l\|)^{M} e^{(\constA-\pi)\sum_{j=1}^l |\Im \l_j|}\,.
\end{equation*}
This implies (\ref{eq:estabphilN}) as $\constA < \pi$.
\end{proof}

The last group of estimates we need shows that if $a\in \mathcal H(\constA,\constP,\delta)$, then there is 
$\varepsilon\in ]0,1]$ so that the function
\begin{equation}
\label{eq:atilde}
\wt a(\l)=\sum_{w\in W} a(w\l) b(w\l)
\end{equation}
belongs to the $W$-invariant Schwartz space $\mathcal S(\mathfrak a_\varepsilon^*)^W$ defined in 
Section \ref{subsection:hypergeomFourier}.

For $0 \leq \eta < 1/2$ set 
\begin{equation}
 \label{eq:TPieta}
T_{\Pi,\eta}=\{\l\in\mathfrak a_\C^*: \text{$|\Re\l_\b|<\tfrac{1}{2}-\eta$ for all $\b \in \Sigma_*^+$}\}\,.
\end{equation}
So $T_{\Pi,0}=T_\Pi$ is the tube domain on which $\Pi( \l)b(\l)$ is holomorphic; 
see Corollary \ref{cor:b-Plancherel-tube-holo}.

\begin{Lemma}
\label{lemma:Schwartztube}
Set $s=|\Sigma_*^+|$.
\begin{enumerate}
\thmlist
\item
Let $0<\eta<1/2$. Then there is a constant $C_\eta>0$ so that
$$|\Pi(\l)b(\l)|\leq C_\eta (1+\|\l\|)^s e^{-\pi\big(\sum_{j=1}^l |\Im \l_j|\big)}$$
for all $\l \in T_{\Pi,\eta}$.
\item
Let $a \in \mathcal H(\constA,\constP,\delta)$ and set $\wt a(\l)=\sum_{w \in W} a(w\l)b(w\l)$. 
Then $\wt a$ is holomorphic in $T_{\Pi} \cap T_\delta$. Moreover, let $0 <\eta <\min\{1/2,\delta\}$. 
Then there are positive constants $C_{\eta,a}>0$ and $C_0$ so that
$$|\wt a(\l)| \leq C_{\eta,a} (1+\|\l\|)^s e^{(\constA-\pi)C_0 \|\Im \l\|}$$
for all $\l\in T_{\Pi,\eta} \cap T_{\delta-\eta}$. 
\item
Let $0 <\eta <\min\{1/2,\delta\}$ and set 
$$\gamma=\min\Big\{\delta-\eta,  \big(\tfrac{1}{2}-\eta\big) \min_{b\in\Sigma_*^+} L_\b^{-1} \Big\}\,.$$
Then $T_\gamma \subset T_{\Pi,\eta} \cap T_{\delta-\eta}$.
Moreover, let $0<\varepsilon <\gamma$. Then $\wt a \in \mathcal S(\mathfrak a^*_\varepsilon)^W$, 
the $W$-invariant Schwartz space on
the tube domain $T_\varepsilon$.
\end{enumerate}
\end{Lemma}
\begin{proof}
Since $\Pi(\l)b(\l)$ is bounded on $T_{\Pi,\eta}$, the proof of the estimate in (a) follows 
the same argument used in part (a) of Lemma \ref{lemma:estimatesTdelta}.

To prove part (b), notice first that, by Corollary \ref{cor:b-Plancherel-tube-holo}, on $T_{\Pi}$
the function $b(\l)$ has at most simple poles on hyperplanes of the form $\l_\b=0$ with 
$\b \in\Sigma_*^+$. The same property holds on $T_{\Pi} \cap T_\delta$ for $b(w\l)a(w\l)$, 
with $w \in W$, and hence for $\wt a(\l)$. 
But $\wt a(\l)$, as a $W$-invariant function, cannot admit first order singularities on root 
hyperplanes through the origin. Thus $\wt a$ is holomorphic on $T_{\Pi} \cap T_\delta$.

Let $0<\eta<\eta'<\min\{1/2,\delta\}$. Choose $C_0>0$ so that $\sum_{j=1}^l |\Im\l_j| \leq C\|\Im\l\|$ for 
all $\l \in \mathfrak a^*$. By (a), there is a constant $C_{\eta'}>0$ so that
$$|\Pi(\l)b(\l)a(\l)| \leq C_{\eta'} (1+\|\l\|)^s e^{(\constA-\pi)\|\Im \l\|_1 }
\leq C_{\eta'}  (1+\|\l\|)^s e^{(\constA-\pi)C_0\|\Im \l\|}$$
for all $\l\in T_{\Pi,\eta'}\cap T_{\delta-\eta'}$.
The required estimate for $\wt a$ on $T_{\Pi,\eta}\cap T_{\delta-\eta}$ follows then as in 
\cite[Lemma 5.6(b)]{OP-Ramanujan-JFA}.

To show that $T_\gamma \subset T_{\Pi,\eta} \cap T_{\delta-\eta}$, notice that $\gamma \leq \delta-\eta$ 
and that if $\l\in T_{\gamma}$, then $|\Re\l_\b|<\gamma L_\b \leq \tfrac{1}{2}-\eta$.

The property that $\wt a \in \mathcal S(\mathfrak a^*_\varepsilon)^W$ for $0<\varepsilon < \gamma$ 
can be proven as in \cite[Lemma 5.6(c)]{OP-Ramanujan-JFA}, using (b) and Cauchy's estimates. 
\end{proof}

Lemma \ref{lemma:Schwartztube} allows us to apply the inversion formula (\ref{eq:inversionsphericalG})
to $\wt a$ and conclude that $\mathcal F^{-1} \wt a \in \mathcal S^p(A)^W \subset (L^p \cap L^2)(A,d\mu)^W$
where $p=2/(\varepsilon+1) \in ]1,2[$. The function $\wt a$ is therefore the hypergeometric Fourier transform 
of $\mathcal F^{-1} \wt a$ and the equality in part (c) of Theorem \ref{thm:RamanujanHO}, initially valid on 
$T_\varepsilon$, extends holomorphically to $T_\Pi \cap T_\delta$. 

We leave the reader to follow the proof of 
\cite[Theorem 2.1, Section 6]{OP-Ramanujan-JFA}, to fill in the missing details of the proof of Theorem 
\ref{thm:RamanujanHO}.

\section{Examples}
\label{section:examples}
Let $(\mathfrak a,\Sigma,m)$ be a triple of rank one with $\Sigma$ of type $A_1$. 
Hence $\Sigma^+=\{\beta\}$ for a unique root $\beta$. 
To simplify notation, we shall write $m$ instead of $m_\beta$. 
Recall from Example \ref{ex:Frankone-compact} the identification of the polynomials $F_{n\b+\rho}(t)$, 
where $t=\exp H \in T$,  with the symmetric Jacobi 
polynomials $X_n^{(m-1)/2}(x)=\hyper{-n}{n+m}{(m+1)/2}{(1-x)/2}$ with $x=\cos\b(H)$. 
By (\ref{eq:d}), 
$$
d(n\b)=\frac{1}{c(n\b+\rho)c^*(-n\b-\rho)}= C_d \big(n+\tfrac{m}{2}\big) \frac{\Gamma(n+m)}{\Gamma(n+1)}\,,
$$  
where $C_d=\wt c(\rho)\wt c\,^*(-\rho)=2/\Gamma(m+1)$.
We identify $\mathfrak a \equiv \mathfrak a^* \equiv \C$ as in Example \ref{ex:Frankone-compact}. 
This means that $\l \in \mathfrak a^*_\C$ and $H \in \mathfrak a_\C=\mathfrak a\oplus \mathfrak t$ 
are respectively identified with $\l_\b \in \C$ and $\b(H)/2 \in \C$.

Let $0 \leq \constA < \pi$ and $\constP>0$ be fixed, and set $a(\l)=e^{-(\constP+i\constA)\l}$. Then 
$$
|a(\l)| =e^{-\constP\Re\l+\constA\Im\l} \leq  e^{-\constP\Re\l+\constA|\Im\l|}
$$
for all $\l \in \mathfrak a^*\equiv \C$. Hence $a \in \mathcal H(\constA,\constP,\delta)$ for every value of $m>0$ 
and for all $\delta \in ]0,1]$. 

Consider the alternating normalized Jacobi series (\ref{eq:fFourier}) for $t=\exp H \in T$    
\begin{align*}
f(t)&=\sum_{n=0}^\infty (-1)^n d(n\b) a(n\b+\rho) F_{n\b+\rho}(t)\\
&=C_d \, e^{-(\constP+i\constA)m/2} \sum_{n=0}^\infty  (-1)^n \big(n+\tfrac{m}{2}\big) \, 
\frac{\Gamma(n+m)}{\Gamma(n+1)}\, e^{-(\constP+i\constA)n}
X_n^{(m-1)/2}(\cos\b(H))\,.
\end{align*}
According to \cite[(10)]{Cowgill}, one has for $|\tau|<1$
\begin{equation}
\label{eq:normalizedserie-ex}
\sum_{n=0}^\infty  \big(n+\tfrac{m}{2}\big) \, \frac{\Gamma(n+m)}{\Gamma(n+1)}\, X_n^{(m-1)/2}(x) \tau^n
= \frac{\Gamma(m +1)}{2} \; G(m,x,\tau)
\end{equation}
where
\begin{equation}
G(m,x,\tau)= \frac{1-\tau^2}{(1-2\tau x+\tau^2)^{1+\tfrac{m}{2}}}
\end{equation}
Thus for $t=\exp H \in T$
\begin{equation}
f(t)=e^{-(\constP+i\constA)\frac{m}{2}} G(m,\cos\b(H),-e^{-(\constP+i\constA)})\,.
\end{equation}
The first part of Ramanujan's Master theorem \ref{thm:RamanujanHO} states that the normalized series 
(\ref{eq:normalizedserie-ex}) converges normally to $f$ on compact subsets of $\{z=i\b(H)/2 \in \C:|\Re z|<\constP/2\}$.
Special instances of the above formulas, involving Legendre or Tchebishef polynomials,
are obtained by selecting particular values of $m$; see Example \ref{ex:Frankone-compact}.

Since $\Sigma$ is reduced, we have $b(\l)=C_b [\sin(\pi\l)]^{-1}$, where
$$C_b=\frac{1}{2} \; \frac{\wt c\,^*(-\rho)}{\wt c(\rho)}=\frac{1}{2} \; 
\frac{\Gamma(m)}{\Gamma\big(\tfrac{m}{2}\big)\Gamma\big(1+\tfrac{m}{2}\big)}$$ is as in (\ref{eq:Cb}).
Hence 
\begin{equation}
 \label{eq:ab-ex}
\sum_{w \in W} a(w\l)b(w\l)=2C_b \; \frac{\sinh((\constP+i\constA)\l)}{\sin(\pi \l)}\,.
\end{equation}
Set $u=\b(H)$ for $H \in \mathfrak a$. Then 
\begin{equation}
f(\exp H)=e^{-(\constP+i\constA)\frac{m}{2}}  G(m,\cosh u,-e^{-(\constP+i\constA)})= \frac{2^{-\frac{m}{2}}\sinh(\constP+i\constA)}{(\cosh u+ \cosh(\constP+i\constA))^{\frac{m}{2}+1}}\,.
\end{equation}

The second part of Ramanujan's Master theorem says that for $u \in ]-\constP,\constP[$ 
and $\sigma \in ]-m/2,m/2[$
we have 
\begin{multline} \label{eq:ex-sum}
\frac{2^{-\frac{m}{2}}\sinh(\constP+i\constA)}{(\cosh u+ \cosh(\constP+i\constA))^{\frac{m}{2}+1}}= \\
C_b \int_{\sigma+i\R} \frac{\sinh((\constP+i\constA)\l)}{\sin(\pi \l)} 
\hyper{\tfrac{m+\l}{2}}{\tfrac{m-\l}{2}}{\tfrac{m+1}{2}}{-\sinh^2 u}\; \frac{d\l}{c(\l)c(-\l)}\,,
\end{multline}
where the integral on the right-hand side extends as an even homomorphic function on a 
neighborhood of $\R$ inside $\C$. Notice that the left-hand side is holomorphic near $\R$. The equality 
(\ref{eq:ex-sum}) therefore holds for all $u \in \R$. In (\ref{eq:ex-sum}) we have used the formula for $F_\l(x)$ given  in Example 
\ref{ex:Frankone} together with the duplication formula
\begin{equation}
\label{eq:hyper-duplication}
\hyper{a}{b}{a+b+1/2}{4z(1-z)}=\hyper{2a}{2b}{a+b+1/2}{z}\,.
\end{equation}
See \cite[2.1.5(27)]{Er}.
(Recall also that the $c$-function appearing under the integral sign in fact depends on the parameter $m>0$.)

If we take for instance $\constA=\sigma=0$, then (\ref{eq:ex-sum}) yields
\begin{equation*} 
\frac{2^{-\frac{m}{2}}\sinh\constP}{(\cosh u+ \cosh\constP)^{\frac{m}{2}+1}} 
=C_b \int_{i\R} \frac{\sinh(\constP\l)}{\sin(\pi \l)} \hyper{\tfrac{m+\l}{2}}{\tfrac{m-\l}{2}}{\tfrac{m+1}{2}}{-\sinh^2 u}\; \frac{d\l}{c(\l)c(-\l)}\,,
\end{equation*}
that is (since the Lebesgue measure on $i\R$ is $d\l$, $i\l \in i\R$)
\begin{multline} \label{eq:ex-withzerosbutum}
\frac{2^{-\frac{m}{2}+1}\sinh\constP}{(\cosh u+ \cosh\constP)^{\frac{m}{2}+1}}= \\
\frac{\Gamma(m)}{\Gamma\big(\tfrac{m}{2}\big)\Gamma\big(1+\tfrac{m}{2}\big)}
\int_{-\infty}^{+\infty} \frac{\sin(\constP\l)}{\sinh(\pi \l)} \hyper{\tfrac{m+\l}{2}}{\tfrac{m-\l}{2}}{\tfrac{m+1}{2}}{-\sinh^2 u}    \frac{d\l}{c(i\l)c(-i\l)}\,.
\end{multline}
Suppose moreover that $u=0$. Since $1+\cosh\constP=2\cosh^2(\constP/2)$, we obtain from 
 (\ref{eq:ex-withzerosbutum})
\begin{equation} \label{eq:ex-withzerosbutm}
2^{-m+1} \big[\cosh(\constP/2)\big]^{-m} \tanh(\constP/2) =
\frac{\Gamma(m)}{\Gamma\big(\tfrac{m}{2}\big)\Gamma\big(1+\tfrac{m}{2}\big)}
\int_{-\infty}^{+\infty} \frac{\sin(\constP\l)}{\sinh(\pi \l)}  \frac{d\l}{c(i\l)c(-i\l)}\,.
\end{equation}
In the case $m=2$,  (\ref{eq:ex-withzerosbutm}) gives
\begin{equation} \label{eq:ex-m=2}
\int_{-\infty}^{+\infty} \frac{\sin(\constP\l)}{\sinh(\pi \l)} \l^2  \; d\l
=\tfrac{1}{2}\, \big[\cosh(\constP/2)\big]^{-2} \tanh(\constP/2).
\end{equation}
Notice that (\ref{eq:ex-m=2}) can also be computed from the classical integral formula
\begin{equation}\label{eq:classicalDwight-one}
\int_{-\infty}^{+\infty} \frac{\sin(\constP x)}{\sinh(\pi x)} \; dx=\frac{1-e^{-\constP}}{1+e^\constP}=\tanh(\constP/2)\,,
\end{equation}
see e.g. \cite[(861.61)]{Dwight}. 
In fact,
$$\int_{-\infty}^{+\infty} \frac{\sin(\constP\l)}{\sinh(\pi \l)} \l^2  \; d\l
=-\frac{d^2}{d\constP^2} \int_{-\infty}^{+\infty} \frac{\sin(\constP\l)}{\sinh(\pi \l)}  \; d\l
=-\frac{d^2  \tanh(\constP/2)}{d\constP^2}\,.$$
For $m=1$, (\ref{eq:ex-withzerosbutm}) yields
\begin{equation} \label{eq:eq:ex-m=1}
\int_{-\infty}^{+\infty} \frac{\sin(\constP\l)}{\sinh(\pi \l)}\,  \l\tanh(\pi\l)  \; d\l=
\big[\cosh(\constP/2)\big]^{-1} \tanh(\constP/2)\,,
\end{equation}
which is again a classical integral formula (see e.g. \cite[(861.81)]{Dwight}).

Finally, the Plancherel formula in the last part of the theorem proves that 
\begin{multline} \label{eq:Plancherel-ex}
\sinh^2(\constP+i\constA) \int_0^\infty \frac{\sinh^m u}{(\cosh u+\cosh(\constP+i\constA))^{m+2}} \; du\\
=\Big[\frac{\Gamma(m)}{2\Gamma\big(\frac{m}{2}\big)\Gamma\big(\frac{m}{2}+1\big)}\Big]^2 
\int_0^\infty \frac{|\sin((\constP+i\constA)\l)|^2}{\sinh^2(\pi \l)} \; \frac{d\l}{c(\l)c(-\l)}\,.
\end{multline}
Moreover, according to (\ref{eq:RamanujanHO}), 
\begin{multline} \label{eq:inversionformulal-ex}
2^{\frac{m}{2}+1} \sinh^2(\constP+i\constA) \int_0^\infty \frac{ \hyper{\tfrac{m+\l}{2}}{\tfrac{m-\l}{2}}{\tfrac{m+1}{2}}{-\sinh^2 u} \sinh^m u }{(\cosh u+\cosh(\constP+i\constA))^{\frac{m}{2}+1}}    \; du=\\
\frac{\Gamma(m)}{\Gamma\big(\frac{m}{2}\big)\Gamma\big(\frac{m}{2}+1\big)} 
 \frac{\sinh((\constP+i\constA)\l)}{\sin(\pi \l)}
\end{multline}
for all $\l \in \C$ with $|\Re\l| < \min\{1/2,m/2\}$.

For $\constA=0$, the right-hand side of (\ref{eq:Plancherel-ex}) can be computed using the integral formula \cite[3.516(3)]{GR}:
\begin{equation*}
\int_0^\infty \frac{\sinh^m u}{(\cosh \constP+\cosh u)^{m+2}} \; du
= \pi^{-\frac{1}{2}} 2^{\frac{m}{2}} e^{-i\frac{m}{2}\,\pi} \frac{\Gamma\big(m+\frac{1}{2}\big)}{\Gamma(m+2)} (\coth^2 \constP-1) Q^{m/2}_{2+m/2}(\coth \constP)\,, 
\end{equation*}
where 
\begin{equation*}
Q^\mu_\nu(z)=\frac{e^{\mu\pi i } \Gamma(\nu+\mu+1) \Gamma(1/2)}{2^{\nu+1} \Gamma(\nu+3/2)} 
\; (z^2-1)^{\mu/2} z^{-\mu-\nu-1} \hyper{\frac{\nu+\mu+2}{2}}{\frac{\nu+\mu+1}{2}}{\nu+\frac{3}{2}}{z^{-2}}
\end{equation*}
is the associated Legendre function of second kind. 

As a special case of (\ref{eq:inversionformulal-ex}) for $\constA=0$ and $m=2$, we notice the formula
\begin{equation} \label{eq:Ramanujanformula-ex}
\frac{1}{\l} \, \int_0^{+\infty} \frac{\sinh(\l u) \tanh u}{(\cosh \constP+\cosh u)^2} \; du=\frac{1}{4\sinh \constP}\;
\frac{\sinh(\constP \l)}{\sin(\pi\l)}\,,
\end{equation}
which is a consequence of 
\cite[2.11(2) and 2.8(12)]{Er} because
\begin{equation*}
\hyper{1+\frac{\l}{2}}{1-\frac{\l}{2}}{\frac{3}{2}}{-\sinh^2 u}=
\hyper{\frac{1}{2}+\frac{\l}{4}}{\frac{1}{2}-\frac{\l}{4}}{\frac{3}{2}}{-\sinh^2 (2u)}=\frac{2\sinh(\l u)}{\l\sinh(2u)}\,.
\end{equation*}

\end{document}